\documentclass[11pt, twoside]{article}
\usepackage[english]{babel}
\usepackage[utf8]{inputenc}
\usepackage[a4paper,top=2.25cm,bottom=2.5cm,left=2.25cm,right=2.25cm]{geometry}
\usepackage{indentfirst}
\usepackage{hyperref}
\usepackage{amsmath,amsfonts,amsthm,mathtools,amssymb,physics}
\usepackage{url}
\usepackage{authblk}
\usepackage{algorithm,algpseudocode}
\usepackage{wrapfig}
\usepackage{tikz,pgfplots,quiver}
\usetikzlibrary{shapes.geometric, arrows.meta, shadings}
\usepgfplotslibrary{groupplots}
\usepackage{enumitem}
\usepackage{graphicx}
\usepackage{subcaption}
\usepackage{float}
\usepackage{mdframed}
\usepackage[capitalize]{cleveref}
\usepackage[acronym]{glossaries}
\usepackage{xcolor}
\hypersetup{
    colorlinks,
    allcolors={black},
}
\usepackage{titlesec}
\titlespacing*{\paragraph}{0pt}{0.4em}{0.4em}

\newcommand{\footnotefirstpage}[1]{
{
\begin{NoHyper}
\renewcommand\thefootnote{}\noindent\footnote{#1}%
\addtocounter{footnote}{-1}%
\end{NoHyper}
}}

\usepackage{mdframed}
\newtheorem{theorem}{Theorem}[section]
\newmdtheoremenv[
    hidealllines=true,
    backgroundcolor=blue!5,
]{boxedtheorem}[theorem]{Theorem}

\newtheorem{proposition}[theorem]{Proposition}

\theoremstyle{definition}\newtheorem{definition}[theorem]{Definition}
\theoremstyle{remark}\newtheorem{remark}[theorem]{Remark}

\newcommand{\R}{\mathbb{R}}
\newcommand{\N}{\mathbb{N}}

\DeclareMathOperator{\Id}{Id}

\newcommand{\seq}{_{N\in\N}}

\renewcommand{\d}{\mathrm{d}}

\newcommand{\supp}{\mathrm{supp}\,}

\newcommand{\prob}{\mathcal{P}}
\newcommand{\probc}{\mathcal{P}_c}

\newcommand{\M}{\mathcal{M}}

\newcommand{\bfmu}{\boldsymbol{\mu}}

\newcommand{\bfrho}{\boldsymbol{\rho}}

\newcommand{\bfy}{\mathbf{y}}
\newcommand{\bfz}{\mathbf{z}}
\newcommand{\bfx}{\mathbf{x}}
\newcommand{\control}{\mathbf{u}}

\newcommand{\bfw}{\mathbf{w}}
\newcommand{\bfv}{\mathbf{v}}

\newcommand{\bfwtau}{\mathbf{w}^{\omega}_\tau}
\newcommand{\bfvtau}{\mathbf{v}^{\omega}_\tau}

\newcommand{\bfp}{\mathbf{p}}
\newcommand{\bfq}{\mathbf{q}}

\newcommand{\bfpsi}{\mathbf{\Psi}}
\newcommand{\bfphi}{\mathbf{\Phi}}
\newcommand{\bftheta}{\mathbf{\Theta}}

\newcommand{\dc}{{c}}

\newcommand{\A}{\mathcal{A}}

\newcommand{\U}{\mathcal{U}}

\newacronym{acr:pmp}{PMP}{Pontryagin Maximum Principle}

\title{Sparse optimal control in the Wasserstein space}

\author[1]{Enrico Sartor}
\author[2]{Florian D\"orfler}
\author[3]{Nicolas Lanzetti}
\affil[1]{Laboratoire des Signaux et Systèmes, Université Paris-Saclay, CentraleSupélec, CNRS\protect\\
\texttt{enrico.sartor@centralesupelec.fr}}
\affil[2]{Automatic Control Laboratory, ETH Zurich\protect\\ \texttt{dorfler@ethz.ch}}
\affil[3]{Department of Computing and Mathematical Sciences, California Institute of Technology\protect\\ \texttt{lnicolas@caltech.edu}}

\begin{document}

\maketitle

\footnotefirstpage{%
This research was supported by the Swiss National Science Foundation under the NCCR Automation, grant agreement 51NF40\_180545. This research was conducted when all authors were affiliated with the Automatic Control Laboratory at ETH Zurich. 
}

\begin{abstract}
    We study sparse optimal control of a non-local continuity equation, where the goal is to steer a distribution via finitely many controllable agents or actuators. This model arises naturally in mean-field multi-agent systems and takes the form of a coupled PDE-ODE system where the PDE describes the evolution of the distribution and the controlled ODE captures the dynamics of the controllable agents. 
A natural objective is distribution steering via terminal costs based on optimal transport, such as the squared Wasserstein distance. These costs are problematic for finite-agent formulations due to non-smoothness at empirical measures and they fall outside common expected-value-type cost classes.
We address these challenges by studying the resulting optimal control problem in the Wasserstein space. Under suitable assumptions on the system dynamics and Wasserstein differentiability of the terminal cost (with no smoothness requirement on the associated Wasserstein gradient), we prove first-order sensitivity of the control-to-state map, derive an adjoint system and an explicit formula for the gradient of the cost functional, and obtain Pontryagin-type necessary conditions. To illustrate the resulting adjoint-based method, we present numerical experiments on a representative distribution-splitting task.

\end{abstract}

\section{Introduction}
This work focuses on the \emph{sparse} optimal control of a non-local continuity equation, whereby the aim is to control a (probability) distribution through a finite number of agents or actuators. 
This model is mainly motivated by mean-field models of large-scale multi-agent systems. 
An emblematic problem instance is the shepherding problem \cite{lien2004shepherding, lama2024shepherding}, in which dogs steer a flock of sheep, while further applications include drone fleet coordination~\cite{perea2009extension}, traffic flow control~\cite{treiber2013traffic, gugat2005optimal}, consensus in distributed systems~\cite{perez2018dynamic}, and evacuation dynamics~\cite{zhou2019guided}; for a broad overview of these problems, we refer to~\cite{chen2019control, piccoli2023control}. In such systems, the focus is on \emph{collective} behavior rather than individual trajectories, and controlling every agent directly is impractical and often impossible, which naturally leads to sparse control formulations over distributions as in~\cite{fornasier2014mean, burger2021mean}.

More formally, we model our problem setting via the coupled ODE-PDE system
\begin{equation}\label{eq:system_dynamics}
    \begin{cases}
        \partial_t\mu(t)+\nabla_x\cdot\bigl(F[\mu(t)](\cdot, y(t))\,\mu(t)\bigr)=0\\
        \dot{y}(t)=G[\mu(t)](y(t))+\control(t),
    \end{cases}
\end{equation}
supplemented with initial conditions $\mu(0)=\mu_0$ and $y(0)=y_0$, where $F$ and $G$ are non-local vector fields and $\control$ is the control input. Here and throughout, dependence on the distribution $\mu$ is denoted by square brackets, the symbols $y$ and $\mu$ denote the state variables (i.e., the unknowns in the ODE--PDE system), and the boldface symbols $\bfy$, $\bfmu$, and $\control$ denote time-dependent trajectories/functions.
The system dynamics~\eqref{eq:system_dynamics} consists of a non-local continuity equation, modeling the time evolution of the distribution $\bfmu(t)$, coupled with an ODE modeling the controllable state $\bfy(t)$.
In our sparse control setting, the control $\control(t)$ acts only on the ODE component. In other words, we cannot \emph{directly} control the distribution $\bfmu(t)$; instead, we influence it \emph{indirectly} through the finite dimensional dynamics of $\bfy(t)$.
For example, in the shepherding problem, $\bfmu(t)$ denotes the distribution of sheep in the plane and $\bfy(t)$ the position of the dogs, on which the control acts directly. The field $F$ encodes the sheep dynamics, both their mutual interactions and their reaction to the dogs, while $G$ specifies the dogs' dynamics in combination with the control input.

The control input is chosen to minimize a general terminal cost $\psi$ that depends on the distribution at some fixed terminal time $T$; e.g., think of splitting a sheep flock and steering it to two distinct points (see~\cref{fig: setting}).
The resulting optimal control problem reads
\begin{equation}\label{eq:control_problem}
\begin{aligned}
    \inf_{\control\in\mathcal{U}}\:\: &\psi\left(\bfmu(T;\control)\right),
\end{aligned}
\end{equation}
where $\bfmu(T;\control)$ denotes the ``measure component'' of the solution of \eqref{eq:system_dynamics} with control $\control$ evaluated at the terminal time $T$ and $\mathcal{U}$ is the set of admissible controls.

A choice of particular interest is $\psi(\mu)=\tfrac{1}{2}W_2(\mu,\hat\mu)^2$, where $W_2$ denotes the $2$-Wasserstein distance and $\hat\mu$ is the desired terminal configuration.
When the objective is to steer the terminal distribution toward $\hat\mu$, $W_2$ is well-suited to this task because it compares probability measures in a permutation-invariant way, it applies to both finite and infinite populations, and, being induced by optimal transport plans, it captures the geometry of the ambient space (unlike $L^p$ norms or statistical divergences).
It can also be viewed as the analog on the space of probability measures of the quadratic costs used in classical optimal control (e.g., linear quadratic regulator). However, the Wasserstein distance, as well as more general optimal transport costs, poses two main difficulties in optimal control. First, when passing to a finite-agent formulation by plugging in empirical measures, the resulting objective is typically non-smooth, since optimal transport plans may be non-unique and optimal transport maps do not always exist at atomic measures. Second, these costs are not of integral (expected-value) form, nor simple functions thereof, and therefore fall outside frameworks tailored to such objectives.

\begin{figure}[t]
        \centering
        \begin{tikzpicture}[>=Stealth, font=\footnotesize, scale=.9]

    \colorlet{densityColor}{blue!50!white}
    \colorlet{leaderColor}{red!80!black}
    \colorlet{targetColor}{black}

    \begin{scope}[shift={(0,0)}]
        \shade[inner color=densityColor, outer color=white] 
            plot [smooth cycle, tension=0.7] coordinates {
                (1.1,0.3) (0.3,0.75) (-0.8,0.4) (-1.2,-0.3) (-0.4,-0.75) (0.9,-0.4)
            };
        \node[blue!80!black] at (0, -1.1) {$\mu(0)$};
        
        \fill[leaderColor] (-1.5, -0.7) circle (2.5pt) node[below, leaderColor] {$y_1(0)$};
        \fill[leaderColor] (1.5, 0.7) circle (2.5pt) node[above, leaderColor] {$y_2(0)$};

        \node[draw, targetColor, fill=targetColor, diamond, inner sep=1.8pt] at (-1.8, 1.0) {};
        \node[above=2pt] at (-1.8, 1.0) {Target 1};
        
        \node[draw, targetColor, fill=targetColor, diamond, inner sep=1.8pt] at (1.8, -1.0) {};
        \node[below=2pt] at (1.8, -1.0) {Target 2};
    \end{scope}

    \draw[->, ultra thick, black] (3.1, 0) -- (4.1, 0);

    \begin{scope}[shift={(7.2,0)}]
        \begin{scope}[shift={(-0.6, 0.25)}]
            \shade[inner color=densityColor, outer color=white] 
                plot [smooth cycle, tension=0.8] coordinates {
                    (-0.3,0.3) (-0.9,0.7) (-1.5,0.5) (-1.2,0.1) (-0.6,0.05)
                };
        \end{scope}
            
        \begin{scope}[shift={(0.6, -0.25)}]
            \shade[inner color=densityColor, outer color=white] 
                plot [smooth cycle, tension=0.8] coordinates {
                    (0.3,-0.3) (0.9,-0.7) (1.5,-0.5) (1.2,-0.1) (0.6,-0.05)
                };
        \end{scope}

        \node[blue!80!black] at (2.5, -0.3) {$\mu(T)$};

        \fill[draw=leaderColor,thick,fill=white] (-1.5, -0.7) circle (2.5pt) node[below, leaderColor] {};
        \fill[draw=leaderColor,thick,fill=white] (1.5, 0.7) circle (2.5pt) node[above, leaderColor] {};
        
        \draw[leaderColor, thick, dashed, opacity=1.0, ->] (-1.5, -0.7) .. controls (-0.8, -0.3) and (-0.6, 0.3) .. (-1.0, 0.6);
        \fill[leaderColor] (-1.0, 0.6) circle (2.5pt) node[right=4.5pt, leaderColor] {$y_1(T)$};
        
        \draw[leaderColor, thick, dashed, opacity=1.0, ->] (1.5, 0.7) .. controls (0.8, 0.3) and (0.6, -0.3) .. (1.0, -0.6);
        \fill[leaderColor] (1.0, -0.6) circle (2.5pt) node[left=3.5pt, leaderColor] {$y_2(T)$};

        \node[draw, targetColor, fill=targetColor, diamond, inner sep=1.8pt] at (-1.8, 1.0) {};
        \node[above=2pt] at (-1.8, 1.0) {Target 1};
        
        \node[draw, targetColor, fill=targetColor, diamond, inner sep=1.8pt] at (1.8, -1.0) {};
        \node[below=2pt] at (1.8, -1.0) {Target 2};
    \end{scope}

\end{tikzpicture}
        \vspace{-2ex}
        \caption{As an example of our problem setting, consider two controlled agents (red circles) that ``split'' a distribution into two of equal mass located at two target locations.}
        \label{fig: setting}
        \vspace{-2ex}
\end{figure}
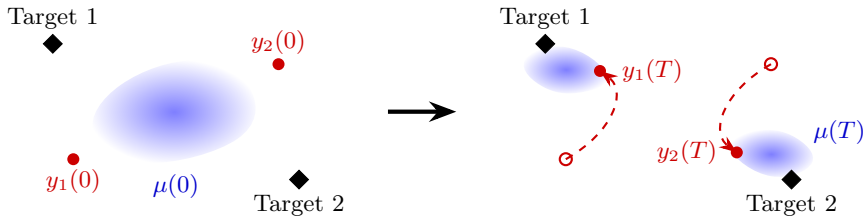

\subsection{Main contributions}

In this paper, we perform the first-order variational analysis directly in the Wasserstein space $\prob_2(\R^d)$ and derive a \emph{smooth} Pontryagin Maximum Principle for the optimal control problem~\eqref{eq:control_problem}. Under standard regularity assumptions on the dynamics, the only differentiability requirement on the terminal cost is \emph{Wasserstein differentiability at the optimal terminal measure}. This yields both a characterization of optimal controls and an explicit adjoint-based formula for the gradient of the cost functional, which we exploit numerically.

More precisely, our contribution is twofold.

\vspace{0.15cm}

\begin{mdframed}[hidealllines=true,backgroundcolor=blue!5]
\textbf{C1.} We derive, in \cref{thrm: PMP}, a \gls{acr:pmp} for the optimal control problem~\eqref{eq:control_problem} in the Wasserstein space.
Our analysis requires only \emph{Wasserstein differentiability of the terminal cost} at the optimal terminal measure (and no additional regularity such as continuity of gradient representatives or barycentric selections). This is the natural analogue of the finite-dimensional smooth \gls{acr:pmp} hypothesis, which assumes differentiability of the terminal cost at the optimal terminal state and nothing more. To the best of our knowledge, this is the first smooth \gls{acr:pmp} in the Wasserstein space that does not require additional regularity conditions on the terminal cost. In particular, our result covers terminal objectives of distribution steering type at the continuum level, including $\psi(\mu)=\frac{1}{2} W_2(\mu,\hat\mu)^2$. Additionally, in \Cref{sec:mean field splitting}, we revisit consistency with particle approximations, including empirical initial measures.
\end{mdframed}

\vspace{0.15cm}

To address the difficulties posed by optimal transport terminal costs, in \Cref{sec:differentiation}, we perform the sensitivity analysis directly at the level of measures and work with the linearized PDE--ODE dynamics in the natural Hilbert setting $L^2(\R^d,\R^d;\mu_0)\times\R^c$ rather than using spaces of continuous functions.
This analysis yields a well-posed adjoint system, which we use to derive, in \Cref{section: optimal control}, a smooth \gls{acr:pmp} under suitable differentiability assumptions on the dynamics and on the terminal cost.
We conduct our analysis within the formalism of Wasserstein calculus. In distribution steering, meaningful perturbations of the state are precisely those that transport mass, and the Wasserstein geometry is designed to encode such transport perturbations; in particular, the perturbations induced by classical needle-like variations of the control are of this form.
While our exposition emphasizes the sparse-control structure, the same sensitivity/adjoint arguments extend to general mean-field optimal control problems in the Wasserstein space, where the control may enter both the transport field and the finite-dimensional drift.

\vspace{0.15cm}

\begin{mdframed}[hidealllines=true,backgroundcolor=blue!5]
\textbf{C2.}
In \cref{thrm: PMP} we obtain an adjoint-based gradient formula for the cost functional with respect to the control which we use to implement a gradient-based algorithm to compute (locally) optimal controls. 
We showcase it on a \emph{distribution splitting} objective, where the goal is to split a distribution into several clusters. This problem is challenging in the finite-agent formulation, since multi-target costs naturally involve an (optimal) assignment/partition of agents, yielding a non-smooth terminal criterion.
By encoding the target through a Wasserstein terminal cost, the assignment is handled implicitly and the resulting continuum objective fits our differentiability framework.
We illustrate numerically our method on a two-dimensional distribution splitting task where an initial population must be steered toward two equal-mass clusters. 
\end{mdframed}

\vspace{0.15cm}

To compute gradients, we need to ensure compatibility between standard G\^ateaux derivatives and Wasserstein differentials. 
Our derivation combines classical $L^2$ calculus with Wasserstein calculus and leads to a three-step procedure: a forward solve of the coupled PDE--ODE system, the computation of a Wasserstein gradient of the terminal cost via an optimal transport problem, and a backward solve of the adjoint PDE--ODE system.
We numerically implement all three steps by discretizing time and space, resulting in a computational cost that depends on the chosen resolution  and the dimension of underlying space, and not on the number of (non-controllable) agents in an underlying particle interpretation, which we use for our two-dimensional distribution splitting in \Cref{subsec:case study}.

\subsection{Related work}
Our problem setting relates to the literature on mean-field (sparse) optimal control and optimal control of non-local continuity equations. 

The sparse mean-field control paradigm, where one acts on a distinguished subset of agents while the remaining population is described through a mean-field limit, was introduced in~\cite{fornasier2014mean} and develops earlier sparse-control ideas for alignment models such as Cucker--Smale~\cite{caponigro2013sparse}.
Building on this line, the authors of~\cite{mfoc} combine a $\Gamma$-convergence argument with a mean-field limit of the adjoint system to derive a Pontryagin-type principle for sparse mean-field dynamics of the form~\eqref{eq:system_dynamics}.
Their analysis is primarily tailored to sufficiently regular running costs and relies on a limiting procedure from finite-dimensional problems, hence yielding optimality conditions naturally stated for controls arising as limits of particle-level optimal controls. Moreover, as discussed above, optimal-transport terminal costs do not directly fit this smooth particle-level route, since the induced finite-dimensional objective obtained by plugging empirical measures is typically non-smooth. By contrast, our approach is formulated directly at the level of continuity equations and provides first-order sensitivity and Pontryagin-type conditions for terminal costs under weak differentiability assumptions compatible with Wasserstein objectives.
A related viewpoint is developed in~\cite{burger2021mean}, where the leaders are represented through continuous trajectories, and the optimality system is obtained via a KKT formulation; see also~\cite{burger2016controlling} for numerical illustrations.
Compared to~\cite{burger2021mean}, we work with an explicit controlled drift in the coupled PDE--ODE dynamics~\eqref{eq:system_dynamics} (rather than taking the leaders' trajectory as the control variable) and focus on terminal costs of low regularity, including Wasserstein distances. For the non-sparse optimal control of non-local continuity equations, we also mention~\cite{chertovskih2023optimal}, where, among other assumptions, the terminal cost is required to admit a locally Lipschitz derivative, a condition that typically fails for optimal transport cost.

Wasserstein costs also appear in the context of pattern formation.
For instance, in~\cite{kreusser2021mean} the Wasserstein distance is used to promote convergence toward a prescribed stationary distribution, leading to a problem formulation over long time horizons and with time-independent controls.
Our focus is instead on finite-horizon (sparse) steering with terminal objectives and on first-order optimality conditions that remain applicable under minimal regularity assumptions on the terminal cost. 

From a technical standpoint, our analysis builds upon the \gls{acr:pmp} in the Wasserstein space pioneered by Bonnet~\cite{bonnet2019pontryagin, nocws, pmpws} for optimal control of continuity equations, which studies optimality conditions using the formalism of subdifferential calculus in Wasserstein spaces~\cite{ambrosiogradientflows,lanzetti,lanzetti2024variational}.
Compared to~\cite{pmpws,nocws,bonnet2019pontryagin}, we (i) incorporate a coupled PDE--ODE structure motivated by sparse leader--follower models, (ii) relax the terminal-cost assumptions to Wasserstein differentiability at the optimal terminal measure, without requiring continuity of gradient representatives or barycentric minimal selections (thus covering squared Wasserstein distances), and (iii) provide an adjoint-based gradient method together with numerical demonstrations.

Complementary to the \gls{acr:pmp} approach, mean-field type control problems can be studied through the dynamic programming principle, leading to Hamilton--Jacobi--Bellman equations on $\prob_2(\R^d)$.
This line of work relies on viscosity solution theories in Wasserstein spaces; see, e.g., \cite{gangbo2008hamilton,pham2018bellman, aussedat2024viscosity, bertucci2024stochastic}.
Within the same dynamic-programming paradigm, \cite{terpin2024dynamic} treats a discrete-time multi-agent problem with individually actuated, uncoupled dynamics and optimal transport costs, proving a separation principle.

\subsection{Notation}
We denote by $\prob(\R^n)$ the set of probability measures over $\R^n$, by $\probc(\R^n)$ the set of $\mu\in\prob(\R^n)$ whose support, $\supp\mu$, is compact, and by
$\prob_2(\R^n)\coloneqq\{\mu\in\prob(\R^n)\colon \int_{\R^n}\norm{x}^2\dd\mu(x) <+\infty\}$
the set of probability measures over $\R^n$ with finite second moment.
For $\mu\in\prob(\R^n)$, we denote by $L^2(\mathbb R^n,\mathbb R^k;\mu)$ and $L^\infty(\mathbb R^n,\mathbb R^k;\mu)$
the spaces of (Borel) measurable functions $f:\mathbb R^n\to\mathbb R^k$, identified up to
$\mu$-a.e.\ equality, such that $\|f\|_{L^2(\mathbb R^n,\mathbb R^k;\mu)}
\coloneqq\bigl(\int_{\mathbb R^n}\|f(x)\|^2\,\dd\mu(x)\bigr)^{1/2}<\infty$ and
$\|f\|_{L^\infty(\mathbb R^n,\mathbb R^k;\mu)}
\coloneqq\operatorname*{ess\,sup}^{\mu}_{x\in\mathbb R^n}\|f(x)\|<\infty$.
We define the \emph{pushforward} of $\mu$ through $f$ as the probability measure $f_\#\mu(C)\coloneqq \mu \bigl(f^{-1}(C)\bigr)$ for every Borel set $C\subseteq \R^n$, and the convolution of $f$ and $\mu$ as the function $f\ast\mu\colon \R^n\to\R^n$ such that
$f\ast\mu(x)\coloneqq \int_{\R^n}f(y-x)\d\mu(y)$
If $X$ is a Banach space, $\mathcal L (X)$ denotes the space of linear and continuous operators from $X$ to $X$.
Finally, we use boldface symbols to refer to functions of time, such as solutions of differential equations and controls.

\subsection{Background material}

We briefly recall background material in optimal transport and differentiability in the Wasserstein space. The material is based on \cite[Part 2]{ambrosiogradientflows} and \cite[Section 2]{lanzetti}.

\paragraph{Wasserstein distance}
The $2$-\emph{Wasserstein distance} between $\mu,\nu\in\prob_2(\R^n)$ is defined as
\begin{equation}\label{eqn: W_2}
    W_2(\mu,\nu)
    \coloneqq
    \inf_{\gamma\in\Gamma(\mu,\nu)}\left(\int_{\R^n\times\R^n}\norm{x-y}^2\d\gamma(x,y)\right)^{\frac{1}{2}},
\end{equation}
where $\Gamma(\mu,\nu)$ is the set of couplings between $\mu$ and $\nu$.
A \emph{coupling} between $\mu$ and $\nu$ is a probability measure $\gamma\in\prob(\R^n\times\R^n)$ such that $\pi^1_\#\gamma=\mu$ and $\pi^2_\#\gamma=\nu$, where $\pi^1,\pi^2\colon\R^n\times\R^n\to\R^n$ are the usual projections over the first and second component, respectively.
We will denote by $\Gamma^o(\mu,\nu)$ the set of such optimal couplings, which is always non-empty as the infimum in \eqref{eqn: W_2} is attained for all $\mu,\nu\in\prob_2(\R^n)$. If $\mu$ is absolutely continuous, then the optimal coupling in~\eqref{eqn: W_2} is unique and induced by a unique optimal transport map $\mathcal T:\R^n\to\R^n$, so that $\gamma=(\Id\times \mathcal T)_\#\mu$ where $\Id$ is the identity map. 
It is well known that the Wasserstein distance is indeed a distance on $\prob_2(\R^n)$ and $\bigl(\prob_2(\R^n), W_2\bigr)$ is a complete metric space, also called \emph{Wasserstein space}. 

\paragraph{Differentiable structure of $\prob_2(\R^n)$}
The \emph{tangent space} at $\mu$ is defined as 
\begin{equation}\label{eq:tangent_space}
    \mathrm{Tan}_\mu\prob_2(\R^n)\coloneqq \overline{\{\nabla\psi\colon \psi\in C_c^\infty(\R^n)\}}^{L^2(\R^n,\R^n;\mu)},
\end{equation}
where $C_c^\infty(\R^n)$ is the space of real-valued smooth functions with compact support. 
With this, Wasserstein differentiability is defined as follows:
\begin{definition}[Wasserstein differentiability]\label{definition: Wasserstein differentiability}
We say that $\varphi\colon \prob_2(\R^n)\to\R$ is \emph{Wasserstein differentiable} at $\mu$ if there exists $\xi\in \mathrm{Tan}_\mu\prob_2(\R^n)$ such that
\begin{equation}\label{eqn: Wasserstein gradients}
    \varphi(\nu)=\varphi(\mu)+\int_{\R^n\times\R^n} \langle \xi(x),y-x\rangle \dd\gamma(x,y)+o\bigl(W_2(\mu,\nu)\bigr)
\end{equation}
for every $\nu\in\prob_2(\R^n)$ and every optimal coupling $\gamma\in\Gamma^o(\mu,\nu)$. It can be shown that such $\xi$ is unique in $\mathrm{Tan}_\mu\prob_2(\R^n)$.
We denote it as $D_\mu\varphi[\mu]$ and we call it \emph{Wasserstein gradient} of $\varphi$ in $\mu$.
We say that $\varphi\colon \prob_2(\R^n)\to\R^k$ is Wasserstein differentiable at $\mu$ if so are its components
$\varphi_1,\ldots,\varphi_k$. In this case, we set 
\begin{equation*}
    D_\mu\varphi[\mu](x)
    \coloneqq 
    \bigl(D_\mu\varphi_1[\mu](x),\ldots, D_\mu\varphi_k[\mu](x)\bigr)^\top
    \in\R^{k\times n},
\end{equation*}
so that $D_\mu\varphi[\mu]\in L^2(\R^n,\R^{k\times n};\mu)$, and it acts on displacements by the usual
matrix-vector product.
\end{definition}
For $F\colon \prob_2(\R^n)\times\R^m\to\R^k $ and $z\in\R^m$ such that the map $\nu\mapsto F[\nu](z)$ is Wasserstein differentiable at $\mu$, we denote its Wasserstein gradient by $x\mapsto D_\mu F[\mu]_z(x)$.
Moreover, Wasserstein gradients are ``strong'': we have an analog of \eqref{eqn: Wasserstein gradients} even for non-optimal couplings.

\begin{proposition}\label{proposition: Wasserstein gradients are strong}
If $\varphi\colon \prob_2(\R^n)\to\R$ is Wasserstein differentiable at $\mu$ then for every $\nu$ we have that
\begin{equation*}
    \varphi(\nu)=\varphi(\mu)+\int_{\R^n\times\R^n}\langle D_\mu\varphi[\mu](x),y-x\rangle \dd\gamma(x,y)+o\Biggl(\biggl(\int_{\R^n\times\R^n}\norm{x-y}^2\dd\gamma(x,y)\biggr)^{\frac{1}{2}}\Biggr)
\end{equation*}
for every $\gamma\in\Gamma(\mu,\nu)$.
\end{proposition}
To conclude, we present two important differentiable functionals: expected values and the Wasserstein distance. A proof of these statements can be found in~\cite[§2]{lanzetti}.

\begin{proposition}\label{proposition: differentiability of functionals}
\,
\begin{enumerate}
\item 
If $\hat\psi\in C^1(\R^d)$ and has at most quadratic growth (i.e., $\psi(x)\leq C(1+\norm{x}^2)$ for some $C>0$) then the map $\mu\mapsto \int_{\R^d}\hat\psi(x)\d\mu(x)$ is Wasserstein differentiable at all $\mu\in\prob_2(\R^d)$ with Wasserstein gradient $\nabla_x\hat\psi$.

\item 
Let $\hat\mu\in\prob_2(\R^d)$ be a reference probability measure. Then, if for $\mu\in\prob_2(\R^d)$ there is a unique optimal transport map $\mathcal T_{\mu}^{\hat\mu}$ from $\mu$ to $\hat\mu$, the map $\mu\mapsto\frac{1}{2}W_2(\mu,\hat\mu)^2$ is Wasserstein differentiable at $\mu$, and its Wasserstein gradient is $\Id-\mathcal T_{\mu}^{\hat\mu}$. This is in particular always true when $\mu$ is absolutely continuous with respect to the Lebesgue measure.
\end{enumerate}
\end{proposition}

\section{The controlled dynamics: well-posedness and continuous dependence}

In this section, we introduce the coupled PDE--ODE system that will serve as the
state equation for our optimal control problem and collect its basic
well-posedness and stability properties.
We consider the control system
\begin{equation}\label{eqn: coupled PDE-ODE system}
\begin{cases}
\partial_t \mu(t) + \nabla_x \cdot \bigl(F[\mu(t)](\cdot,y(t))\,\mu(t)\bigr) = 0,\\[2pt]
\dot y(t) = G\bigl[\mu(t)\bigr](y(t))+\control(t),
\end{cases}
\end{equation}
where $\mu(t)\in\probc(\R^d)$,
$y(t)\in\R^c$, and $\control\in\U \coloneqq L^2([0,T],U)$ is a control law that takes values in a fixed compact and convex set $U\subset\R^c$. 

\subsection{Well-posedness of the controlled system}

We first introduce the notion of weak measure solution for the coupled
PDE--ODE system~\eqref{eqn: coupled PDE-ODE system}.

\begin{definition}[Weak measure solution]\label{def: solution}
We say that a pair $(\bfmu,\bfy):[0,T]\to\probc(\R^d)\times\R^c$ is a weak measure
solution of~\eqref{eqn: coupled PDE-ODE system} if
\begin{enumerate}
  \item $\bfmu$ is equi-compactly supported, i.e., there exists $R>0$ such that
        $\supp\bfmu(t)\subset B(0,R)$ for all $t\in[0,T]$;
  \item for every test function $\varphi\in C_c^\infty((0,T)\times \R^d)$ it holds that
  \begin{equation}\label{eqn: weak formulation}
  \int_0^T\int_{\R^d}
  \bigl(\partial_t\varphi(t,x)
        + \nabla_x\varphi(t,x)\cdot F[\bfmu(t)](x,\bfy(t))\bigr)\dd(\bfmu(t))(x)\dd t
  = 0;
  \end{equation}
  \item $\bfy$ is a Carathéodory solution of $\dot y(t)=G[\bfmu(t)](y(t))+\control(t)$.
\end{enumerate}
A Cauchy problem for~\eqref{eqn: coupled PDE-ODE system} is specified by an initial
time $t_0\in[0,T]$, an initial probability measure $\mu_0\in\probc(\R^d)$, and an
initial controlled state $y_0\in\R^c$. We say that $(\bfmu,\bfy)$ is a solution of
the Cauchy problem with data $(t_0,\mu_0,y_0,\control)$ if it is a weak measure solution in the sense above and satisfies $\bfmu(t_0)=\mu_0$ and $
\bfy(t_0)=y_0$.
\end{definition}

The previous definition is Eulerian in nature: the continuity equation is
formulated in terms of the time evolution of the measure-valued curve
$t\mapsto\mu(t)$ on $\probc(\R^d)$. Under suitable regularity assumptions on the
vector field $F$, this Eulerian description can be complemented with a
Lagrangian one in terms of a \emph{flow} of characteristics, as we will make
precise below. To do so, we start with regularity assumptions on $F$ and $G$ that will be used throughout the paper.

\noindent\emph{Lipschitz continuity.} There exists $\mathbf{L}>0$ such that for all $t\in[0,T]$, all
$\mu_1,\mu_2\in\probc(\R^d)$, $x_1,x_2\in\R^d$, and $y_1,y_2\in\R^c$,
\begin{equation}\label{eqn: lipschitz continuity}\tag{\bf LIP}
    \begin{split}
        \norm{F[\mu_1](x_1,y_1)-F[\mu_2](x_2,y_2)}
        &\leq \mathbf{L}\bigl(\norm{x_1-x_2}+\mathfrak{d}\left((\mu_1,y_1),(\mu_2,y_2)\bigr)\right),\\[2pt]
        \norm{G[\mu_1](y_1)-G[\mu_2](y_2)}
        &\leq \mathbf{L} \mathfrak{d}\left((\mu_1,y_1),(\mu_2,y_2)\right),
    \end{split}
\end{equation}
where $\mathfrak{d}\bigl((\mu_1,y_1),(\mu_2,y_2)\bigr)\coloneqq W_2(\mu_1,\mu_2)+\norm{y_1-y_2}$ is the product metric on $\prob_2(\R^d)\times \R^c$.
\noindent\emph{Sublinear growth.} There exists $\mathbf{C}>0$ such that for all $t\in[0,T]$, all $\mu\in\probc(\R^d)$,
$x\in\R^d$, and $y\in\R^c$,
\begin{equation}\label{eqn: sublinear growth}\tag{\bf SLG}
    \begin{split}
        \norm{F[\mu](x,y)}
        &\leq \mathbf{C}\bigl(1+\norm{x}+\norm{y}+m_\infty(\mu)\bigr),\\[2pt]
        \norm{G[\mu](y)}
        &\leq \mathbf{C}\bigl(1+\norm{y}+m_\infty(\mu)\bigr),
    \end{split}
\end{equation}
where $m_\infty(\mu)$ is the (finite) radius of the support of $\mu$, i.e., $m_\infty(\mu)\coloneqq \sup\bigl\{\norm{x}\colon x\in\supp\mu\bigr\}$.

The following theorem provides existence, uniqueness, and a Lagrangian
representation of solutions of~\eqref{eqn: coupled PDE-ODE system}.
These fundamental properties will be used repeatedly throughout the paper.

\begin{theorem}[Well-posedness and flow representation]\label{thrm: well posedness}
Assume that~\eqref{eqn: lipschitz continuity} and~\eqref{eqn: sublinear growth} hold. Then, for every control law
$\control\in\U$, initial time $t_0\in[0,T]$, initial distribution $\mu_0\in\probc(\R^d)$,
and initial controlled state $y_0\in\R^c$, there exists a unique solution, in the sense
of~\cref{def: solution}, 
\begin{equation*}
    t\longmapsto\bigl(\bfmu(t;t_0,\mu_0,y_0,\control),\,\bfy(t;t_0,\mu_0,y_0,\control)\bigr)
    \equiv 
    \bigl(\bfmu(t;\control),\,\bfy(t;\control)\bigr)
\end{equation*}
of the corresponding Cauchy problem associated with~\eqref{eqn: coupled PDE-ODE system}, where the dependence on $t_0,\mu_0,y_0$, and eventually $\control$, is omitted when clear from the context. 
Moreover, the solution enjoys the following
properties:
\begin{enumerate}
    \item[(i)] \textbf{A priori bounds.}  
    If $R>0$ is such that $m_\infty(\mu_0)\leq R$ and $\norm{y_0}\leq R$
    then, for every $\control\in\U$ and every
    $t\in[0,T]$, one has
    \begin{equation*}
          m_\infty\bigl(\bfmu(t;\control)\bigr),
    \norm{\bfy(t;\control)}\leq  \bar R ,
    \end{equation*}
    where $\bar R=\bar R(R, U, \mathbf{C}, T)>0$ depends only on $R$, $U$, $\mathbf{C}$, and $T$.
    \item[(ii)] \textbf{Lagrangian representation.}  
    For every control law $\control\in\U$, every $\mu_0\in\probc(\R^d)$, $y_0\in\R^c$, and
    starting time $t_0\in[0,T]$, there exists a (non-local) flow
    \begin{equation*}
    \bfphi(\cdot,\cdot;\control)\equiv \bfphi(\cdot,\cdot; t_0,\mu_0,y_0,\control)\colon[0,T]\times\R^d\to\R^d
    \end{equation*}
    such that $\bfphi(t_0, x;\control) = x$, for all $x\in\R^d$ and all $t\in[0,T]$
    \begin{equation*}
      \frac{\dd}{\dd t}\bfphi(t, x; \control)
    =
    F\bigl[\bfmu(t;\control)\bigr]\bigl(
        \bfphi(t, x;\control),
        \bfy(t; \control)\bigr)
    \end{equation*}
    and $\bfmu(t;\control)=\bfphi(t,\cdot;\control)_\#\mu_0$. In particular, if $\mu_0$ is absolutely continuous with respect to the Lebesgue measure, then for every $t\in[0,T]$ so is $\bfmu(t;\control)$.
\end{enumerate}
\end{theorem}

The proof of~\cref{thrm: well posedness} follows the standard arguments for
Carathéodory differential equations~\cite[Chapter 2]{imtc} and non-local
continuity equations with Lipschitz vector fields~\cite{piccoli2013transport}, and is
therefore omitted. A similar result can be found in a related mean-field optimal control
setting in~\cite{fornasier2014mean}. Moreover, the
a priori growth bounds in~\cref{thrm: well posedness} show that the assumption
\eqref{eqn: lipschitz continuity} only needs to hold
on bounded sets. Indeed, by the estimate in~(i), all trajectories starting from
$m_\infty(\mu_0)\le R$ and $\norm{y_0}\le R$ remain in a ball whose radius depends only on $R$, $U$, $\mathbf{C}$,
and $T$, so global assumptions can be replaced by local ones on such balls. To avoid overburdening the notation, we nevertheless prefer to work in the globally Lipschitz setting throughout.

\subsection{Continuous dependence}

We now investigate the dependence of the solution of \eqref{eqn: coupled PDE-ODE system} on the initial conditions and on the control laws.  Throughout, we assume that \eqref{eqn: lipschitz continuity} and \eqref{eqn: sublinear growth} hold.

Our first result, whose proof is a standard application of the Grönwall lemma, shows Lipschitz dependence on initial data.

\begin{proposition}\label{theorem: Lipschitz dependence on initial data}
For every $R>0$ there exists $L=L(R,\mathbf{C}, U, T)>0$ such that, if $(\bfmu_1,\bfy_1)$ and $(\bfmu_2,\bfy_2)$ are solutions of \eqref{eqn: coupled PDE-ODE system} with control $\control$ and initial conditions $(\mu_{0,1},y_{0,1})$ and $(\mu_{0,2},y_{0,2})$ at time $t_0$, satisfying $m_\infty(\mu_{0,1}),m_\infty(\mu_{0,2}),\norm{y_{0,1}},\norm{y_{0,2}}\leq R$, then the estimate
\begin{equation}\label{eqn: continuous dependence, thesis}
    \mathfrak{d}\bigl((\bfmu_1(t),\bfy_1(t)),(\bfmu_2(t),\bfy_2(t))\bigr)\leq L\mathfrak{d}\bigr((\mu_{0,1},y_{0,1}),(\mu_{0,2},y_{0,2})\bigr)
\end{equation}
holds for every $t\in[0,T]$.
\end{proposition}

We now show that solutions depend continuously on controls.

\begin{theorem}[Continuous dependence]\label{theorem: continuous dependence on controls}
Let $(\control_N)\seq\subset \U$ be such that $\control_N\rightharpoonup \control$ weakly in
$L^2([0,T],\R^c)$. Let $(\mu_0^N,y_0^N)\seq$ be initial conditions at time $t_0\in[0,T]$ such that $\mu_0^N$ converges in $W_2$ to
$\mu^0\in\probc(\R^d)$ and $y_0^N\to y^0$ in $\R^c$. Assume that there exists $R>0$ such that $\supp \mu_0^N\subset B(0,R)$ and $\norm{y_0^N}\leq R$ for every $N\in\N$. Let $(\bfmu_N,\bfy_N)$ be the (unique) solution of \eqref{eqn: coupled PDE-ODE system}
associated with the control $\control_N$ and initial condition $(\mu_0^N,y_0^N)$, and let $(\bfmu,\bfy)$ be
the (unique) solution associated with $(\control,\mu^0,y^0)$. Then $(\bfmu_N,\bfy_N)\to(\bfmu,\bfy)$ uniformly
in $[0,T]$ with respect to the distance $\mathfrak{d}$.
\end{theorem}
\begin{proof}
The proof follows the lines of \cite[Theorem 3.3]{fornasier2014mean}, and is therefore just sketched. Uniform a priori estimates give equiboundedness and uniform Lipschitz continuity of $(\mu_N,y_N)$. Hence, by Ascoli--Arzel\`a, a subsequence converges uniformly in time to some $(\bfmu,\bfy)$. Passing to the limit in the integral formulation is then straightforward: it relies on the Lipschitz continuity of $F$ and $G$, together with the convergence of the control terms. Indeed, the dynamics depends linearly on the control and $\control_N\rightharpoonup \control$ in $L^2$, so in particular $\int_0^t \control_N(s)\dd s \to \int_0^t \control(s)\dd s$ for every $t\in[0,T]$. Thus, $(\bfmu,\bfy)$ solves \eqref{eqn: coupled PDE-ODE system} with limit data. By uniqueness, the whole sequence converges.
\end{proof}

\section{Differentiation of the dynamics}\label{sec:differentiation}

In this section, we study the differentiability properties of the solutions of \eqref{eqn: coupled PDE-ODE system} with respect to the initial conditions and controls. 
These results are key for the computation of gradients and the derivation of the Pontryagin Maximum Principle in the next section. Since we believe that they are of independent interest, we present them in a dedicated section.
Henceforth, we assume that $F$ and $G$ satisfy both \eqref{eqn: lipschitz continuity} and \eqref{eqn: sublinear growth} so that the coupled system~\eqref{eqn: coupled PDE-ODE system} is well posed.

Let us now fix an initial time $t_0\in[0,T]$, initial reference probability measure $\mu_0\in\probc(\R^d)$ and position $y_0\in \R^\dc$ and a control law $\control^\star\in\mathcal{U}$. Under these conditions, \cref{thrm: well posedness} ensures the existence of a unique probability measure curve $\bfmu^\star\colon [t_0,T]\to\probc(\R^d),$ with corresponding non-local flow $\bfphi^\star\colon [t_0,T]\times \R^d\to\R^d$, and controllable trajectory $\bfy^\star\colon [t_0,T]\to\R^c$ that solve \eqref{eqn: coupled PDE-ODE system}. We refer to $(\bfmu^\star,\bfy^\star)$ as the reference trajectory.
Then, for $\varepsilon>0$, we consider perturbed initial conditions $\mu_0^\varepsilon\in \probc(\R^d)$ and $y_0^\varepsilon\in\R^c$ having the form
\begin{equation}\label{eqn:perturbation}
    \mu^\varepsilon_0 = (\mathrm{Id}+\varepsilon w_0)_\#\mu_0\quad\text{ and }\quad y_0^\varepsilon = y_0+\varepsilon v_0,
\end{equation}
for $w_0\in L^\infty(\R^d,\R^d;\mu_0)$ and $v_0\in\R^c$, as well as a perturbed control $\control^\varepsilon = \control^\star +\varepsilon \delta\control$
where $\delta \control\in L^2(0,T;\R^c)$ is such that $\control^\varepsilon\in \U$. Since $U$ is compact, this forces in particular $\delta \control\in L^\infty(0,T;\R^c)$.
With these perturbed data, we have a unique solution $(\bfmu^\varepsilon,\bfy^\varepsilon)$, where the probability distribution can be expressed as $    \bfmu^\varepsilon(t) = \bfphi^\varepsilon(t,\cdot)_\#\mu_0^\varepsilon$,
with $\bfphi^\varepsilon\colon[t_0,T]\times\R^d\to\R^d$ being the non-local flow corresponding to the initial conditions $(t_0,\mu_0^\varepsilon,y_0^\varepsilon)$ and control $\control^\varepsilon$. 

Because the tangent spaces to $\prob_2(\mathbb R^d)$ depend on the base point (being the closure with respect to $L^2(\R^d,\R^d;\mu)$ of gradients of compactly supported functions, see~\eqref{eq:tangent_space}), 
the linearized dynamics cannot be written in a single fixed Hilbert space, unless we transport tangent vectors to a common reference measure. To this end, we use 
the flow $\bfphi^\star$ to ``pull 
back'' vectors to the initial space $L^2(\R^d,\R^d;\mu_0)$. This naturally leads us to 
formulate the linearized system in the Hilbert space $H(\mu_0) \coloneqq L^2(\R^d,\R^d;\mu_0)\times \mathbb R^c$, whose elements we denote as $\mathfrak{w}\coloneqq (w,v)$, equipped with the inner product
\begin{equation*}
    \bigl((w_1,v_1),(w_2,v_2)\bigr)_{H(\mu_0)}
    \coloneqq 
    \int_{\mathbb R^d} w_1(x)\cdot w_2(x) \dd\mu_0(x) + v_1\cdot v_2.
\end{equation*}
We denote by $\norm{\cdot}_{H(\mu_0)}$ the corresponding norm.
On this space, we introduce the family of linear operators $\A(t)\colon H(\mu_0) \to H(\mu_0) $ for $t\in[0,T]$
\begin{equation*}
    \A(t) = \bigl(\A_w(t), \A_v(t)), \qquad \A_w(t)\colon H(\mu_0)\to L^2(\R^d,\R^d;\mu_0), \quad \A_v(t)\colon H(\mu_0)\to \R^c,
\end{equation*}
where
\begin{align*}
    \bigl(\A_w(t) \mathfrak{w}\bigr)(x) &=  D_xF[\bfmu^\star(t)](\bfphi^\star(t,x),\bfy^\star(t))\, w(x)
    + D_yF[\bfmu^\star(t)](\bfphi^\star(t,x), \bfy^\star(t))\, v\\
    &+\int_{\R^d}\bigl(D_\mu F[\bfmu^\star(t)]_{(\bfphi^\star(t,x),\bfy^\star(t))}(\bfphi^\star(t,\zeta))\bigr)\, w(\zeta)\dd\mu_0(\zeta)
\end{align*}
and
\begin{align*}
    \A_v(t)\mathfrak{w}=&D_yG[\bfmu^\star(t)](\bfy^\star(t))\, v
    +\int_{\R^d}\bigl(D_\mu G[\bfmu^\star(t)]_{\bfy^\star(t)}(\bfphi^\star(t,\zeta))\bigr)\, w(\zeta)\dd\mu_0(\zeta).
\end{align*}
Informally, these operators are obtained by differentiating the vector fields $F$ and $G$ with respect to the state variables $x,y$ and the distribution $\mu$, where we use the flow $\bfphi^\star$ to express all integrals with respect the initial probability measure $\mu_0$ instead of $\bfmu^\star(t)$.
We now consider the linearized equation along the reference trajectory, which gives the linear inhomogeneous system
\begin{equation}\label{eqn: abstract linearized system}
    \begin{cases}
        \dot{\mathfrak{w}}(t) = \A(t) \mathfrak{w}(t)+\overline{\delta \control}\\
        \mathfrak{w}(t_0) = \mathfrak{w}_0 \coloneq (w_0, v_0),
    \end{cases}
    \qquad 
    \text{where}
    \qquad 
    \overline{\delta \control}=\bigl(0_{L^2(\R^d,\R^d;\mu_0)}, \delta \control\bigr),
\end{equation}
whose initial condition is given precisely by the ``vectors'' $w_0$ and $v_0$ that induce the $\varepsilon$-perturbations.
Given that we restrict to bounded perturbations (i.e., $w_0\in L^\infty(\R^d,\R^d;\mu_0)$ and $v_0\in\R^c$), it is instrumental to introduce the Banach space
\begin{equation*}
    X(\mu_0) \coloneqq L^\infty(\R^d,\R^d;\mu_0)\times \mathbb R^c, \hspace{1.5cm}\norm{\mathfrak{w}}_{X(\mu_0)}=\norm{w}_{L^\infty(\R^d,\R^d;\mu_0)}+\norm{v},
\end{equation*}
that embeds continuously into $H(\mu_0)$.
The goal of this section is to prove the following result, which, informally, quantifies the effect of the perturbation on the terminal cost via solutions of \eqref{eqn: abstract linearized system}.

\vspace{0.15cm}

\begin{boxedtheorem}[Differentiable dependence]\label{thrm: differentiability}
Assume that $\psi\colon\prob_2(\R^d)\to\R$ is Wasserstein differentiable at
$\bfmu^\star(T)$, and that $F$ and $G$ satisfy the following regularity
assumptions:
\begin{enumerate}[label=\textnormal{H\arabic*)}]
  \item $F$ and $G$ are differentiable in all their variables;
  \item the gradients in $x$ and $y$ are continuous,
  \begin{equation*}
    (\mu,x,y)\mapsto D_{x,y}F[\mu](x,y),
    \qquad
    (\mu,y)\mapsto D_y G[\mu](y);
  \end{equation*}
  \item the Wasserstein differentials are continuous,
  \begin{equation*}
    (\mu,x,y,\zeta)\mapsto D_\mu F[\mu]_{(x,y)}(\zeta),
    \qquad
    (\mu,y,\zeta)\mapsto D_\mu G[\mu]_{y}(\zeta).
  \end{equation*}
\end{enumerate}
Then, if $\mathfrak{w}_0\in X(\mu_0)$, for the solutions $(\bfmu^\varepsilon,\bfy^\varepsilon)$ associated with
the perturbed data $(\mu_0^\varepsilon,y_0^\varepsilon,\control^\varepsilon)$,
we have
\begin{equation*}
  \lim_{\varepsilon\to 0}
  \frac{\psi\bigl(\bfmu^\varepsilon(T)\bigr)
        - \psi\bigl(\bfmu^\star(T)\bigr)}{\varepsilon}
  = \int_{\R^d}
      D_\mu\psi[\bfmu^\star(T)]
      \bigl(\bfphi^\star(T,x)\bigr)\cdot \mathbf w(T,x)\,d\mu_0(x),
\end{equation*}
where $\mathfrak w=(\mathbf w,\mathbf v)$ is the unique solution of
\eqref{eqn: abstract linearized system}.
\end{boxedtheorem}

\vspace{0.15cm}

The proof requires some intermediate steps, which we present next, and is thus deferred to~\Cref{subsec:proof:differentiability}. 

\begin{remark}
Perturbations on probability measures induced by transport maps, such as those in \eqref{eqn:perturbation}, are relevant in distribution steering problems and, as we shall see below, cover the ones induced by needle variations in the proof of the \gls{acr:pmp}.
Accordingly, the Wasserstein geometry, and thus Wasserstein calculus, is natural for the linearization of~\eqref{eqn: coupled PDE-ODE system} in $\prob_2(\R^d)$. 
In contrast, extrinsic notions of differentiation (e.g., flat derivatives) do not naturally encode transport-induced variations.
\end{remark}

\subsection{The linearized dynamics}

To start, we study the well-posedness and main properties of the linearized system \eqref{eqn: abstract linearized system} and introduce its evolution family.

\begin{proposition}\label{propo: linear system}
The linear inhomogeneous system \eqref{eqn: abstract linearized system} admits a unique solution
\begin{equation}\label{eqn: representation formula for the linear system}
    \mathfrak{w}(t)=\M(t,t_0) \mathfrak{w}_0+\int_{t_0}^t \M(t,\tau) \overline{\delta\control}(\tau) d\tau
\end{equation}
where $\M$ is the evolution family associated to the homogeneous part of \eqref{eqn: abstract linearized system}.
Moreover, the subspace $X(\mu_0)$ is invariant; i.e., 
if $\mathfrak w_0\in X(\mu_0)$
then the solution $\mathfrak w$ of \eqref{eqn: abstract linearized system}
satisfies $\mathfrak w\in C([t_0,T];X(\mu_0))$.
\end{proposition}
\begin{proof}
Since $\supp\mu_0$ is compact and $(t,x)\mapsto \bfphi^\star(t,x)$ is continuous,
all coefficients appearing in $\A(t)$ are bounded on compact sets by assumptions
\textnormal{H2)--H3)}. This implies $\A(t)\in\mathcal L(H(\mu_0))$ and continuity
of $t\mapsto \A(t)$ in operator norm, so well-posedness for the homogeneous part follows from
\cite[Theorem~5.2]{pazy2012semigroups}. The inhomogeneous part is integrable from $[t_0,T]\to X(\mu_0)$ so the Duhamel principle yields \eqref{eqn: representation formula for the linear system}. The same boundedness yields that $\A(t)$
maps $X(\mu_0)$ into itself, hence $X(\mu_0)$ is invariant.
\end{proof}

Next, we linearize the system dynamics along the reference trajectory. 
In the previous part, we showed that the linearized system \eqref{eqn: abstract linearized system} admits a unique solution $(\bfw,\bfv)$. Therefore, we can consider the corresponding perturbed maps
\begin{equation*}
    \bfpsi^\varepsilon\coloneqq \bfphi^\star+\varepsilon \bfw\quad \text{ and }\quad \bfz^\varepsilon\coloneqq \bfy^\star+\varepsilon \bfv
\end{equation*}
and curve of probability measures $\bfrho^\varepsilon(t)\coloneqq \bfpsi^\varepsilon(t,\cdot)_\#\mu_0$.
Then, the vector fields $F$ and $G$ admit the following expansion along this perturbation.

\begin{proposition}\label{propo: first order expansion}
Let $\bfrho^\varepsilon$, $\bfpsi^\varepsilon$, $\bfz^\varepsilon$ and
$\mathfrak w=(\bfw,\bfv)$ be defined as above. Then there exist remainder terms
$R_F^\varepsilon\colon[t_0,T]\to L^2(\R^d,\R^d;\mu_0)$ and
$R_G^\varepsilon\colon[t_0,T]\to\R^c$ such that, for every
$s\in[t_0,T]$ and $\mu_0$-a.e.\ $x\in\R^d$,
\begin{align*}
  F[\bfrho^\varepsilon(s)]\bigl(\bfpsi^\varepsilon(s,x),\bfz^\varepsilon(s)\bigr)
  &= F[\bfmu^\star(s)]\bigl(\bfphi^\star(s,x), \bfy^\star(s)\bigr)
     + \varepsilon\,\bigl(\A_w(s)\mathfrak w(s)\bigr)(x)
     + R_F^\varepsilon(s,x)\\
  G[\bfrho^\varepsilon(s)]\bigl(\bfz^\varepsilon(s)\bigr)
  &= G[\bfmu^\star(s)]\bigl(\bfy^\star(s)\bigr)
     + \varepsilon\,\A_v(s)\mathfrak w(s)
     + R_G^\varepsilon(s).
\end{align*}
Moreover, the remainders are $o(\varepsilon)$ uniformly in $s$, in the sense that
\begin{equation*}
    \sup_{s\in[t_0,T]}
    \frac{\|R_F^\varepsilon(s,\cdot)\|_{L^2(\R^d,\R^d;\mu_0)}}{\varepsilon}
    \longrightarrow 0,
  \qquad
  \sup_{s\in[t_0,T]}
    \frac{\norm{R_G^\varepsilon(s)}}{\varepsilon}
    \longrightarrow 0
  \quad\text{as }\varepsilon\to 0.
\end{equation*}
In particular, the first expansion holds in $L^2(\R^d,\R^d;\mu_0)$ and both remainders are
uniformly $o(\varepsilon)$ in time.
\end{proposition}
\begin{proof}
See Appendix \ref{sec: technical proofs}.
\end{proof}

\subsection{Derivative of the flow}

We can now state the fundamental result on the differentiability of solutions of \eqref{eqn: coupled PDE-ODE system} with respect to initial conditions and controls.

\begin{proposition}\label{theorem: differentiability with respect to initial conditions}
Let $w_0\in L^\infty(\R^d,\R^d;\mu_0)$, $v_0\in\R^c$ and
$\mathfrak w=(\bfw,\bfv)$ be the unique solution of
\eqref{eqn: abstract linearized system} with initial condition
$\mathfrak w(t_0)=(w_0,v_0)$. For $\varepsilon>0$ and $t\in[t_0,T]$ define $\bftheta^\varepsilon(t,x)
  \coloneqq \bfphi^\varepsilon\bigl(t,x+\varepsilon w_0(x)\bigr)$ in $L^2(\R^d,\R^d;\mu_0)$. Then
\begin{equation*}
      \lim_{\varepsilon\to 0}
  \frac{(\bftheta^\varepsilon,\bfy^\varepsilon)
        -(\bfphi^\star,\bfy^\star)}{\varepsilon}
  = \mathfrak w
\end{equation*}
in the uniform topology from $[t_0,T]$ into $H(\mu_0)$.
\end{proposition}
\medskip

\begin{proof}
To keep the notation compact we write time dependence as a subscript. We notice that the definition of $\bftheta^\varepsilon_t$ is well-posed since $\bfphi^\varepsilon_t$ is continuous, $\mu_0$ has compact support and $w_0$ is essentially bounded. For $t\in[t_0,T]$ and $x\in\R^d$ set $\Xi^\varepsilon_t
\coloneqq\bigl\|(\boldsymbol{\Theta}^\varepsilon_t,\bfy^\varepsilon_t)-(\bfpsi^\varepsilon_t,\bfz^\varepsilon_t)\bigr\|_{H(\mu_0)}$ and
\begin{equation*}
    \alpha^\varepsilon_t(x)
    \coloneqq F\bigl[\bfmu_t^\varepsilon\bigr]
                     \bigl(\bftheta_t^\varepsilon(x), \bfy_t^\varepsilon\bigr),
  \quad
  \beta^\varepsilon_t(x)
    \coloneqq F\bigl[\bfrho_t^\varepsilon\bigr]\bigl(\bfpsi^\varepsilon_t(x), \bfz^\varepsilon_t\bigr).
\end{equation*}
Using the integral formulation of the nonlinear flow for $(\boldsymbol{\Theta}^\varepsilon,\bfy^\varepsilon)$ and of the linearized system for
$(\bfpsi^\varepsilon,\bfz^\varepsilon)$ one obtains
\begin{equation*}
    \boldsymbol{\Theta}^\varepsilon_t(x)-\bfpsi^\varepsilon_t(x)
=\int_{t_0}^t\!\bigl(\alpha^\varepsilon_s(x)-\beta^\varepsilon_s(x)\bigr)\dd s
+\int_{t_0}^t\!{R_F^\varepsilon}_s(x)\dd s, \quad \text{in }L^2(\R^d,\R^d;\mu_0)
\end{equation*}
for all $t\in[t_0,T]$, where $R_F^\varepsilon$ is the remainder from \cref{propo: first order expansion}.
By Lipschitz continuity of $F$ and the estimate
$W_2(\bfmu^\varepsilon_s,\bfrho^\varepsilon_s)\leq \|\boldsymbol{\Theta}^\varepsilon_s(\cdot)-\bfpsi^\varepsilon_s(\cdot)\|_{L^2(\R^d,\R^d;\mu_0)}$,
we deduce an inequality of the form
\begin{equation*}
    \|\boldsymbol{\Theta}^\varepsilon_t(\cdot)-\bfpsi^\varepsilon_t(\cdot)\|_{L^2(\R^d,\R^d;\mu_0)}
    \leq o(\varepsilon) + C\int_{t_0}^t \Xi^\varepsilon_s\dd s,
\end{equation*}
uniformly in $t$. A completely analogous estimate for the ODE component, taking into account also the perturbation on the control, yields $    \|\bfy^\varepsilon_t-\bfz^\varepsilon_t\|\leq o(\varepsilon)+ C\int_{t_0}^t \Xi^\varepsilon_s\dd s$.
Combining the two bounds gives $    \Xi^\varepsilon_t\le o(\varepsilon)+ 2C\int_{t_0}^t \Xi^\varepsilon_s\dd s$ for $t\in[t_0,T]$. By Grönwall’s lemma, $\sup_{t\in[t_0,T]}\Xi^\varepsilon_t=o(\varepsilon)$.
Recalling that $(\bfpsi^\varepsilon,\bfz^\varepsilon)=(\bfphi^\star,\bfy^\star)+\varepsilon(\bfw,\bfv)$, this is equivalent to
\begin{equation*}
    \sup_{t\in[t_0,T]}
    \left\|
    \frac{(\boldsymbol{\Theta}^\varepsilon_t,\bfy^\varepsilon_t)-(\bfphi^\star_t,\bfy^\star_t)}{\varepsilon}
    -\mathfrak w_t
    \right\|_{H(\mu_0)}
    \longrightarrow 0,
\end{equation*}
which is the desired differentiability statement.
\end{proof}

\subsection{Proof of~\cref{thrm: differentiability}}\label{subsec:proof:differentiability}
Now, the proof of \cref{thrm: differentiability} follows.
\begin{proof}[Proof of~\cref{thrm: differentiability}]
We recall that by the non-local flow property $\bfmu^\star(T) = \bfphi^\star(T,\cdot)_\#\mu_0$, $\bfmu^\varepsilon(T) = \bfphi^\varepsilon(T,\cdot)_\#\mu_0^\varepsilon$
and that, by definition of the perturbed initial measure, $  \mu_0^\varepsilon = (\mathrm{Id}+\varepsilon w_0)_\#\mu_0$.
Consequently,
\begin{equation*}
\bfmu^\varepsilon(T)
  = \bfphi^\varepsilon(T,\cdot)_\#\mu_0^\varepsilon
  = \bigl(\bfphi^\varepsilon(T,\cdot)\circ(\mathrm{Id}+\varepsilon w_0)\bigr)_\#\mu_0
  = \bftheta^\varepsilon(T,\cdot)_\#\mu_0,
\end{equation*}
where $\bftheta^\varepsilon$ is as in~\cref{theorem: differentiability with respect to initial conditions}.
Thus, the measure
$
  \gamma^\varepsilon
  \coloneqq \bigl(\bfphi^\star(T,\cdot),\bftheta^\varepsilon(T,\cdot)\bigr)_\#\mu_0
$
is a coupling between $\bfmu^\star(T)$ and $\bfmu^\varepsilon(T)$. Moreover,
\begin{align*}
    W_{\gamma^\varepsilon}
    &\coloneqq
    \biggl(\int _{\R^d\times\R^d}\norm{x-y}^2\dd\gamma^\varepsilon(x,y)\biggr)^{\frac{1}{2}}
    =
    \biggl(\int_{\R^d} \norm{\bfphi^\star(T,x)-\bftheta^\varepsilon(T,x)}^2\dd\mu_0(x)\biggr)^{\frac{1}{2}}\\
    &=
    \varepsilon\norm{\bfw(T)}_{L^2(\R^d,\R^d;\mu_0)}
    +
    o(\varepsilon)
    =
    O(\varepsilon)
\end{align*}
by~\cref{theorem: differentiability with respect to initial conditions}, where $\mathfrak w=(\bfw,\bfv)$ is the unique solution of the linearized system \eqref{eqn: abstract linearized system}.
Since $\psi$ is Wasserstein differentiable at $\bfmu^\star(T)$ and Wasserstein
gradients are strong (see~\cref{proposition: Wasserstein gradients are strong}) we have
\begin{equation}\label{eqn: diff-psi-Wasserstein}
  \frac{\psi\bigl(\bfmu^\varepsilon(T)\bigr)-\psi\bigl(\bfmu^\star(T)\bigr)}{\varepsilon}
  = \frac{1}{\varepsilon}
    \int_{\R^d\times\R^d} D_\mu\psi[\bfmu^\star(T)](x)\cdot (y-x)\dd\gamma^\varepsilon(x,y)
    + \frac{o(W_{\gamma^\varepsilon})}{\varepsilon}.
\end{equation}
We now note that $o(W_{\gamma^\varepsilon})=o(\varepsilon)$ and recall the definition of $\gamma^\varepsilon$ to get
\begin{align*}
  \eqref{eqn: diff-psi-Wasserstein}
  &= \int_{\R^d}
    D_\mu\psi[\bfmu^\star(T)]\bigl(\bfphi^\star(T,x)\bigr)\cdot
      \frac{\bftheta^\varepsilon(T,x)-\bfphi^\star(T,x)}{\varepsilon}\dd\mu_0(x)
    + \frac{o(\varepsilon)}{\varepsilon}
    \\
    &=
    \left\langle D_\mu\psi[\bfmu^\star(T)]\bigl(\bfphi^\star(T,\cdot)\bigr),
    \frac{\bftheta^\varepsilon(T,\cdot)-\bfphi^\star(T,\cdot)}{\varepsilon}
  \right\rangle_{L^2(\R^d,\R^d;\mu_0)}
  + \frac{o(\varepsilon)}{\varepsilon}.
\end{align*}

By~\cref{theorem: differentiability with respect to initial conditions}, the second term in the scalar product converges in $L^2(\R^d,\R^d;\mu_0)$, as $\varepsilon\to 0$ to $\bfw(T)$. Since
$D_\mu\psi[\bfmu^\star(T)]\circ\bfphi^\star(T,\cdot)\in L^2(\R^d,\R^d;\mu_0)$, we can let $\varepsilon\to 0$ and obtain
\begin{equation*}
      \lim_{\varepsilon\to 0}
  \frac{\psi\bigl(\bfmu^\varepsilon(T)\bigr)-\psi\bigl(\bfmu^\star(T)\bigr)}{\varepsilon}
  = \int_{\R^d}
      D_\mu\psi[\bfmu^\star(T)]\bigl(\bfphi^\star(T,x)\bigr)\cdot
      \bfw(T,x)\dd\mu_0(x),
\end{equation*}
which is the desired formula.
\end{proof}

\section{The optimal control problem}\label{section: optimal control}

We are now ready to consider the optimal control problem:
\begin{equation}\label{eqn: optimal control problem}\tag{\textbf{OCP}}
    \inf_{\control\in\mathcal{U}}J(\mu_0,y_0,\control)=\inf_{u\in\mathcal{U}}\psi\bigl(\bfmu(T;\mu_0,y_0,\control)\bigr),
\end{equation}
subjected to the dynamical constraint
\begin{equation}\label{eqn: general controlled coupled PDE-ODE system}
    \begin{cases}
        \partial_t\mu(t)+\nabla _x\cdot \bigl(F[\mu(t)](\cdot,y(t))\,\mu(t)\bigr)=0\\
        \dot{y}(t)=G[\mu(t)](y(t))+\control(t)
    \end{cases}
\end{equation}
with initial conditions $\mu(0)=\mu_0$ and $y(0)=y_0$. Under assumptions \eqref{eqn: lipschitz continuity} and \eqref{eqn: sublinear growth}, for every $\control\in\mathcal U$ the system admits a unique solution, hence $J$ is well-defined. 

The main result of this section, a \gls{acr:pmp} for \eqref{eqn: optimal control problem}, provides an explicit formula for the G\^ateaux derivative of the cost functional $J$
with respect to the control, together with first-order optimality conditions for the optimal control problem.

A key tool in the derivation is the \emph{backward adjoint equation}
\begin{equation}\label{eqn: abstract adjoint linearized system}
    \begin{cases}
        -\dot{\mathfrak{p}}(t) =\A(t)^* \mathfrak{p}(t)\\
        \mathfrak{p}(T) = \mathfrak{p}_T,
    \end{cases}
\end{equation}
where $\A(t)^*=(\A_w(t)^*, \A_v(t)^*)$ is the adjoint operator of $\A(t)$ and the terminal condition $\mathfrak{p}_T$ belongs to $H(\mu_0)$.
From a direct computation, the adjoint operator admits the following explicit expressions: for $\mu_0$-a.e.\ $x\in\R^d$,
\begin{align*}
            \bigl(\A_w(t)^*\mathfrak{p}\bigr)(x)&=D_xF[\bfmu^\star(t)](\bfphi^\star(t,x),\bfy^\star(t))^\top p(x)
            +D_\mu G[\bfmu^\star(t)]_{\bfy^\star(t)}(\bfphi^\star(t,x))^\top q\\
            &+\int_{\R^d} D_\mu F[\bfmu^\star(t)]_{(\bfphi^\star(t,\zeta),\bfy^\star(t))}(\bfphi^\star(t,x))^\top p(\zeta) \dd\mu_0(\zeta)
\end{align*}
and
\begin{align*}
    \A_v(t)^*\mathfrak{p} &= \int_{\R^d}D_yF[\bfmu^\star(t)](\bfphi^\star(t,\zeta),\bfy^\star(t))^\top p(\zeta)\dd\mu_0(\zeta)
    +D_yG[\bfmu^\star(t)](\bfy^\star(t))^\top q.
\end{align*}
We are now ready to state our \gls{acr:pmp}.

\vspace{0.15cm}

\begin{boxedtheorem}[\gls{acr:pmp}]\label{thrm: PMP}
If $\psi\colon \prob_2(\R^d)\to\R$ is lower semi-continuous with respect to $W_2$ then there exists an optimal control for \eqref{eqn: optimal control problem}.  
Let $\control^\star\in \U$, $(\bfmu^\star,\bfy^\star)$ and $\bfphi^\star$ be the corresponding solution and flow, respectively,
and assume $\psi$ is Wasserstein differentiable at $\bfmu^\star(T)$. If we consider the unique solution $\mathfrak{p}^\star=(\bfp^\star, \bfq^\star)$ of the adjoint equation \eqref{eqn: abstract adjoint linearized system} with terminal condition
\begin{equation*}
    \mathfrak{p}(T)=\bigl(D_\mu\psi [\bfmu^\star(T)](\bfphi^\star(T,\cdot)), 0),
\end{equation*}
we have that for the perturbed control $\control^\varepsilon=\control^\star+\varepsilon \delta\control\in\U$ it holds
\begin{equation*}
    \psi\bigl(\bfmu^\varepsilon(T))=\psi\bigl(\bfmu^\star(T))+\varepsilon\langle \bfq^\star, \delta \control\rangle_{L^2([0,T],\R^c)}+o(\varepsilon).
\end{equation*}
Finally, if $\control^\star$ is a local minimizer for \eqref{eqn: optimal control problem} in the strong $L^2([0,T],\R^c)$ topology on $\U$ , then the following optimality condition holds for almost every $t\in [0,T]$
\begin{equation*}
    \bfq^\star(t)\cdot \control^\star(t)=\min_{\omega\in U}\bfq^\star(t)\cdot\omega.
\end{equation*}
\end{boxedtheorem}
\vspace{0.15cm}

The proof of~\cref{thrm: PMP} is standard given the machinery developed in the previous section. Specifically, the \gls{acr:pmp} follows the classical needle-variation argument (see, e.g., \cite[Theorem~6.1.1]{imtc}), combined with the sensitivity results developed in the previous section. We detail the proof after discussing our main result.

\subsection{Discussion}\label{subsection: discussion}

A few comments on our results are in order.

First, we observe that the PMP in \cref{thrm: PMP} mirrors the classical finite-dimensional PMP, in that it provides a pointwise-in-time optimality condition involving an adjoint state solving a backward adjoint equation, whose terminal condition is given by the (Wasserstein) gradient of the terminal cost. We emphasize that, within this analogy, no additional regularity of the terminal cost beyond differentiability is required (for instance, the existence of a continuous representative is not needed).

Second, although the exposition focuses on the sparse structure $\dot y = G(\mu,y)+\control$, the proof of~\cref{thrm: PMP} only relies on the regularity assumptions on the terminal cost and on the vector fields ensuring \cref{thrm: differentiability}. Consequently, the same arguments apply to more general controlled dynamics, where the control may enter (linearly, if one wishes to carry over the existence proof verbatim) both the transport field $F$ and the drift $G$. In particular, this covers the non sparse settings in \cite{pmpws, nocws}.

Third, under the regularity assumptions of~\cref{theorem: differentiability with respect to initial conditions}, $X(\mu_0)$ is invariant also for the backward adjoint equation. Thus, if we assume that $D_\mu\psi [\bfmu^\star(T)]\bigl(\bfphi^\star(T,\cdot)\bigr)\in L^\infty(\R^d,\R^d;\mu_0)$ we obtain that $\mathbf{p}^\star(t)\in L^\infty(\R^d,\R^d;\mu_0)$ for every $t\in[0,T]$ and so $\boldsymbol{\nu}\coloneqq (\bfphi^\star(t,\cdot), \bfp(t,\cdot))_\#\mu_0$ is a curve of compactly supported probability on the cotangent bundle $\R^d\times\R^d$. This provides us with the Hamiltonian interpretation of the dynamics as $(\boldsymbol{\nu},\bfy,\bfq)$ solves the system
\begin{equation}\label{eqn: phase space coupled system}
    \begin{cases}
        \partial_t\nu(t)+\nabla_{(x,p)}\cdot\bigl(\mathcal{H}[\nu(t)](\cdot,y(t),\cdot,q(t))\nu(t)\bigr)=0\\
        \dot{y}(t)=\control^\star(t)\\
        -\dot{q}(t)=\int_{\R^{d}\times\R^d}D_y F[\pi^1_\#\nu(t)](x,y(t))^\top p \,\dd\nu(t)(x,p),
        \end{cases}
\end{equation}
with boundary conditions $\pi^1_\#\nu(0)=\mu_0, y(0)=y_0, 
        \pi^2_\#\nu(T)=D_\mu\psi[\bfmu^\star(T)](\cdot)_\#\bfmu^\star(T)$, and $q(T)=0$,
where $\mathcal{H}=(\mathcal{H}_x,\mathcal{H}_p)$ with
\begin{align*}
    \mathcal{H}_x[\nu](x,y,p,q)=&F[\pi^1_\#\nu](x,y)\\
    \mathcal{H}_p[\nu](x,y,p,q)=&-D_xF[\pi^1_\#\nu](x,y)^\top p
    -\int_{\R^d\times\R^d} D_\mu F[\pi^1_\#\nu]_{(\xi,y)}^\top(x) r\dd\nu(\xi,r)
\end{align*}
and we assumed $G\equiv 0$ to simplify the notation. We thus recover the formulations in \cite{bongini2017mean, pmpws} as special cases. 
In particular, our analysis reveals that boundedness assumptions are only needed to linearize the system dynamics~\eqref{eqn: coupled PDE-ODE system} but not for the adjoint system~\eqref{eqn: abstract adjoint linearized system} that naturally lives in the larger space $H(\mu_0)$.

Fourth, from a numerical viewpoint, directly solving the optimality system in~\cref{thrm: PMP} amounts, after space discretization, to solving a high-dimensional two-point boundary value problem; consequently, (multiple) shooting techniques are difficult to implement in practice. A gradient-descent approach is instead simpler to implement. Nonetheless, the PMP provides a useful a posteriori criterion to assess the quality of the computed solution. We follow this strategy in the case study of~\Cref{sec:mean field splitting}.

\subsection{Proof of existence}
Existence of optimal controls for~\eqref{eqn: optimal control problem} follows by the direct method. 
Since $\mathcal U$ is weakly sequentially compact in $L^2$ (by compactness and convexity of $U$), any minimizing sequence $(\control_n)_n\subset\mathcal U$ admits a subsequence, still denoted $(\control_n)_n$, such that $\control_n\rightharpoonup \control^\star$ in $L^2$. 
By~\cref{theorem: continuous dependence on controls}, the solution map $\control\mapsto \bfmu(\mu_0,y_0,\control;T)$ is weakly continuous; in particular, $\bfmu(T;\control_n)$ converges to $\bfmu(T;\control^\star)$ in $W_2$.
Since $\psi$ is lower semicontinuous with respect to $W_2$, $\control^\star$ is optimal.

\subsection{Proof of the $\varepsilon$-expansion}

In this part, we derive the gradient formula. Given that the terminal cost $\psi$ is Wasserstein differentiable and the initial conditions are not perturbed, \cref{thrm: differentiability} yields
\begin{align*}
\lim_{\varepsilon\to0}\frac{\psi(\bfmu^\varepsilon(T))-\psi(\bfmu(T))}{\varepsilon}
&=\int_{\R^d} D_\mu\psi[\bfmu(T)](\bfphi(T,x)) \cdot \bfw(T,x) \dd\mu_0(x)\\
&=\Bigl\langle \mathfrak p(T),\int_0^T \mathcal M(T,\tau)\,\overline{\delta\control}(\tau)\,d\tau\Bigr\rangle_{H(\mu_0)}\\
&=\int_0^T \bigl\langle \mathcal M(T,\tau)^*\mathfrak p(T),\,\overline{\delta\control}(\tau)\bigr\rangle_{H(\mu_0)}\,d\tau\\
&=\int_0^T \bigl\langle \mathfrak p(\tau),\,\overline{\delta\control}(\tau)\bigr\rangle_{H(\mu_0)}\dd\tau
=\int_0^T \bfq(\tau)\cdot\delta\control(\tau)\dd\tau,
\end{align*}
where we used Fubini's theorem to exchange the integrals and the representation of the solutions of \eqref{eqn: abstract adjoint linearized system} via the adjoint evolution family. 

\subsection{Proof of the \gls{acr:pmp}}

We fix a locally (in $L_2$) optimal control $\control^\star$ for \eqref{eqn: optimal control problem} with respect to the initial conditions $\mu_0\in\probc(\R^d)$ and $y_0\in\R^c$,  and consider the corresponding optimal trajectory $(\bfmu^\star,\bfy^\star)\colon [0,T]\to \prob_c(\R^d)\times\R^c$ and flow $\bfphi^\star\colon[0,T]\times\R^d\to\R^d$.

To derive a first-order optimality condition for
\eqref{eqn: optimal control problem} we perturb the optimal control $\control^\star$. Fix $\tau\in(0,T]$ and let $\bfphi^\star_\tau$ the flow starting at time $\tau$.
For $\omega\in U$ and $\varepsilon>0$ sufficiently small, we define the \emph{needle variation}
\begin{equation*}
    \control^\varepsilon(t)\coloneqq
    \begin{cases}
        \omega & \text{if }t\in[\tau-\varepsilon,\tau],\\
        \control^\star(t) & \text{otherwise}.
    \end{cases}
\end{equation*}
We denote by $(\bfmu^\varepsilon,\bfy^\varepsilon)$ the corresponding solutions. Since
$\control^\varepsilon=\control^\star$ on $[0,\tau-\varepsilon)$, we have that $    (\bfmu^\varepsilon(t),\bfy^\varepsilon(t))=(\bfmu^\star(t),\bfy^\star(t))$ for all $t\in[0, \tau-\varepsilon]$. In particular, letting $\bfphi_{\tau-\varepsilon}^\varepsilon$ be the flow starting at
$\tau-\varepsilon$ from $(\bfmu^\star(\tau-\varepsilon),\bfy^\star(\tau-\varepsilon))$ driven by
$\control^\varepsilon$, we can write
\begin{equation}\label{eqn: needle variation initial condition}
    \bfmu^\varepsilon(\tau)
    =\Bigl(\bfphi_{\tau-\varepsilon}^\varepsilon(\tau,\cdot)\circ \bfphi^\star_\tau(\tau-\varepsilon,\cdot)\Bigr)_\#\bfmu^\star(\tau).
\end{equation}
Since $\control^\varepsilon\to\control^\star$ strongly in $L^2([0,T],\R^c)$ (hence also weakly),
Theorem~\ref{theorem: continuous dependence on controls} yields that
$(\bfmu^\varepsilon,\bfy^\varepsilon)\to(\bfmu^\star,\bfy^\star)$ on $[\tau-\varepsilon,\tau]$.
Moreover, by compactness of the supports and continuity of $F$, the corresponding vector fields are bounded and uniformly continuous on a common compact region containing the relevant probability measure trajectories, flow characteristics and $y$ trajectories. Therefore, with $F_\tau(x)\coloneqq F[\bfmu^\star(\tau)]\bigl(x,\bfy^\star(\tau)\bigr)$, by the integral formulation of the flow, we have the first-order expansions
\begin{align*}
    \bfphi^\star_\tau(\tau-\varepsilon,x)& = x - \int_{\tau-\varepsilon}^\tau F[\bfmu^\star(t)]\bigl(\bfphi^\star_\tau(t,x), \bfy^\star(t)\bigr)\,\dd t = x - \varepsilon F_\tau(x)+o(\varepsilon),\\
    \bfphi_{\tau-\varepsilon}^\varepsilon (\tau, x) &= x+\int_{\tau-\varepsilon}^\tau F[\bfmu^\varepsilon(t)]\bigl(\bfphi_{\tau-\varepsilon}^\varepsilon(t, x), \bfy^\varepsilon(t)\bigr)\,\dd t = x+ \varepsilon F_\tau(x)+o(\varepsilon),
\end{align*}
uniformly for $x\in\supp\bfmu^\star(\tau)$ and $x\in\supp\bfmu^\star(\tau-\varepsilon)$, respectively. Indeed, by uniform continuity of $F$ on the relevant compact set, there exists $\omega_\varepsilon\to0$ such that $$\sup_x\Bigl|\int_{\tau-\varepsilon}^{\tau}\bigl(F[\bfmu^\star(t)](\bfphi^\star_\tau(t,x),\bfy^\star(t))-F_\tau(x)\bigr)\,\dd t\Bigr|
\le \varepsilon\,\omega_\varepsilon=o(\varepsilon),$$
and analogously for the $\varepsilon$-perturbed flow.
Thus, since $x\mapsto F_\tau(x)$ is $C^1$, composing the two expansions we get $\bfmu^\varepsilon(\tau)=(\mathrm{Id}+w^\varepsilon)_\# \bfmu^\star(\tau)$ with $\norm{w^\varepsilon}_{L^2(\R^d,\R^d;\bfmu^\star(\tau))}=o(\varepsilon)$.
Finally, from a standard application of the Lebesgue differentiation theorem in PMP proofs
(see, e.g., \cite[Thm.~6.1.1]{imtc}) we have, for a.e.\ $\tau\in[0,T]$, $\bfy^\varepsilon(\tau)=\bfy^\star(\tau)+\varepsilon\bigl(\omega-\control^\star(\tau)\bigr)+o(\varepsilon)$.

Fix a $\tau$ for which the previous expansion holds. Since $\control^\varepsilon$ and $\control^\star$ coincide on $[\tau,T]$, the perturbation acts,
from time $\tau$ onward, only through the shifted initial conditions~\eqref{eqn: needle variation initial condition}.
In particular, in the linearized system~\eqref{eqn: abstract linearized system} at time $\tau$ we have
$\delta\control\equiv 0$, and the initial data are
\begin{equation}\label{eqn: initial conditions at tau}
    w(\tau)=0,
    \qquad
    v(\tau)=\omega-\control^\star(\tau).
\end{equation}
By local optimality of $\control^\star$, since $\control^\varepsilon\to\control^\star$ strongly in $L^2([0,T],\R^c)$ , for all $\varepsilon>0$ small enough,
\begin{equation*}
        \frac{\psi(\bfmu^\varepsilon(T))-\psi(\bfmu^\star(T))}{\varepsilon}\geq 0.
\end{equation*}
Applying Theorem~\ref{thrm: differentiability} and passing to the limit $\varepsilon\to 0$ gives
\begin{equation}\label{eqn: limit optimality condition}
    \int_{\R^d}
        D_\mu\psi[\bfmu^\star(T)]\bigl(\bfphi_\tau^\star(T,x)\bigr)\cdot \bfwtau(T,x)\,
    \dd(\bfmu^\star)(\tau)(x)\geq 0,
\end{equation}
where $(\bfwtau,\bfvtau)$ is the solution of the linearized system in $H(\bfmu^\star(\tau))$ with initial conditions \eqref{eqn: initial conditions at tau}. Let $\mathfrak{p}_\tau^\star=(\bfp_\tau^\star,\bfq_\tau^\star)$ be the solution in $H(\bfmu^\star(\tau))$ of the
backward adjoint equation~\eqref{eqn: abstract adjoint linearized system} driven by the flow $\bfphi_\tau^\star$ starting at $\tau$, with terminal condition $\mathfrak{p}^\star_\tau(T)=\bigl(D_\mu\psi [\bfmu^\star(T)](\bfphi^\star_\tau(T,\cdot)), 0)$. By the adjointness relation, the map
\begin{equation*}
    t\mapsto \int_{\R^d} \bfp_\tau^\star(t,x) \cdot\bfwtau(t,x) \dd(\bfmu^\star(\tau))(x)+
    \bfq_\tau^\star(t)\cdot\bfvtau(t)
\end{equation*}
is constant. Thus, \eqref{eqn: limit optimality condition}  tells us that at the terminal time $T$ it is non-negative. As a consequence at time $\tau$, recalling \eqref{eqn: initial conditions at tau},  we have 
\begin{equation*}
    \int_{\R^d}\bfp^\star_\tau(\tau,x)\cdot\bfwtau(\tau,x)\dd(\bfmu^\star(\tau))(x)+\bfq^\star(\tau)\cdot \bfvtau(\tau)=\bfq^\star_\tau(\tau)\cdot(\omega-\control^\star(\tau))\geq 0.
\end{equation*}
It remains to identify the adjoint $\mathfrak{p}_\tau^\star$ in $H(\bfmu^\star(\tau))$ with the adjoint
$\mathfrak{p}^\star$ in $H(\mu_0)$ appearing in~\cref{thrm: PMP}. This follows from the flow composition property and
the change-of-variables formula for pushforward measures. More precisely, since
$\bfmu^\star(\tau)=(\bfphi^\star(\tau,\cdot))_\#\mu_0$, one readily checks that for $t\in[\tau,T]$ and $\mu_0$-all $x\in\R^d$
\begin{equation*}
    \bfp^\star(t,x)=\bfp_\tau^\star\bigl(t,\bfphi^\star(\tau,x)\bigr),\qquad\bfq^\star(t)=\bfq_\tau^\star(t).
\end{equation*}
The claim follows from uniqueness, since both pairs solve~\eqref{eqn: abstract adjoint linearized system} on $[\tau,T]$ with the same terminal condition. Therefore, $\bfq^\star_\tau(\tau)\cdot(\omega-\control^\star(\tau))\geq 0$. Given that the choice of $\omega$ was arbitrary and that \eqref{eqn: needle variation initial condition} holds for almost every $\tau\in[0,T]$ this readily implies that for almost every $t\in [0,T]$ it holds $\bfq^\star(t)\cdot \control^\star(t)=\min_{\omega\in U}\bfq^\star(t)\cdot\omega$.

\section{A mean-field splitting problem}\label{sec:mean field splitting}

In this section, we introduce a microscopic $N$-agent model describing a population of ``passive'' agents interacting with a small number of controlled leaders. The control objective is to split the population distribution into a prescribed finite number of clusters of equal mass, each centered at a designated target location (see \cref{fig: setting} for the two-cluster case). This formulation is natural in applications such as swarm robotics, biology, and the social sciences.
For large $N$, this system is well approximated by the mean-field PDE--ODE dynamics studied in the previous sections. Our theoretical analysis and numerical optimization are carried out at the mean-field level, using a finite-difference discretization of the PDE. In a final step, we validate the resulting mean-field optimal controls by applying them to large but finite agent systems.

\subsection{Model}
We consider, in $\R^d$, a system of $N$ non-controllable agents and $M$ controllable ones over a fixed horizon $[0,T]$. For $N\in\N$ and $n=1,\dots,N$ we denote by $\bfx^{N,n}(t)\in\R^d$ the position of the $n$-th passive agent and set
\begin{equation*}
        \bfx^N(t)\coloneqq \bigl(\bfx^{N,1}(t),\dots,\bfx^{N,N}(t)\bigr)\in (\R^d)^N.
\end{equation*}
The controllable agents have positions $\bfy_m(t)\in\R^d$, $m=1,\dots,M$, collected in $\bfy(t)\in (\R^d)^M$, and are driven by controls $\control_m\colon[0,T]\to U_m$, where $U_m\subset\R^d$ is compact and convex. The interactions among agents are described by kernels $K,f, g\colon\R^d\to\R^d$, which we assume to be odd. The dynamics are given by
\begin{equation}\label{eqn:finite-dimensional-equations}
    \begin{cases}
        \displaystyle
        \dot{x}^{N,n}(t)
        =\frac{1}{N}\sum_{i=1}^N K\bigl(x^{N,i}(t)-x^{N,n}(t)\bigr)
         +\frac{1}{M}\sum_{j=1}^M f\bigl(y_j(t)-x^{N,n}(t)\bigr),
        & \!n=1,\dots,N\\[0.5em]
        \displaystyle
        \dot{y}_m(t)=\frac{1}{M}\sum_{j=1}^Mg\bigl(y_j(t)-y_m(t)\bigr)+\control_m(t), & \!m=1,\dots,M
    \end{cases}
\end{equation}
with initial data $\bfx_0^N\coloneqq \bigl(x^{N,1}_0,\dots,x^{N,N}_0\bigr)$ and $\bfy_0\coloneqq \bigl(y_0^1,\dots,y_0^M\bigr)$.
Define the empirical measure of the non-controllable agents by
\begin{equation*}
        \bfmu^N(t)\coloneqq \frac{1}{N}\sum_{n=1}^N \delta_{\bfx^{N,n}(t)}\in\prob_c(\R^d).
\end{equation*}
It is well known (see, e.g., \cite{burger2021mean}) that $(\bfmu^N,\bfy)$ solves, in the sense of \cref{def: solution}, the coupled PDE–ODE system
\begin{equation}\label{eqn:mean-field-coupled-system}
    \begin{cases}
        \partial_t \mu(t)
        + \nabla_x\cdot\bigl(F[\mu(t)](\cdot,y(t))\,\mu(t)\bigr)=0,\\[0.3em]
        \dot{y}(t)=G[\mu(t)](y(t))+\control(t),
    \end{cases}
\end{equation}
where $\control=(\control_1,\dots,\control_M)$,
\begin{equation*}
        F[\mu](x,y)
    \coloneqq K\ast\mu(x)+\frac{1}{M}\sum_{m=1}^M f\bigl(y_m-x\bigr) \quad \text{ and }\quad G[\mu](y)_m=\frac{1}{M}\sum_{j=1}^Mg(y_j-y_m).
\end{equation*}
With respect to the notation used in the previous sections, we have therefore $c=dM$ and $U=U_1\times...\times U_M$. Under the assumptions of~\cref{theorem: continuous dependence on controls}, if $\mu_0^N\coloneqq \bfmu^N(0)\to\mu_0^\infty$ in $W_2$ then the corresponding solutions $(\bfmu^N,\bfy^N)$ converge uniformly in time to the solution $(\bfmu_\infty,\bfy_\infty)$ of \eqref{eqn:mean-field-coupled-system} starting from $(\mu_0^\infty, y_0)$. This provides a justification of the continuum (mean-field) approximation for large finite-dimensional systems once we assume $\mu_0^\infty$ to be absolutely continuous with respect to the Lebesgue measure on $\R^d$.

Our goal is to split the ensemble of non-controllable agents into $P$ clusters of the same mass at given positions $z_1,...,z_P\in\R^d$. 
Despite its wide applicability, this problem poses significant challenges. 
Indeed, the terminal cost cannot be expressed as a standard expected value but must be expressed as a Wasserstein distance to the desired configuration $\hat\mu\coloneqq\frac{1}{P}\sum_{i=1}^P\delta_{z_i}$ with $z_i\in\R^d$ for $i=1,..,P$. This leads to the optimal control problem
\begin{equation}\label{eqn: finite dimensional OPC}\tag{\ensuremath{\mathbf{OCP_N}}}
    \inf_{\control\in \mathcal{U}}J_N(\control)
    =
    \inf_{\control\in \mathcal{U}}\frac{1}{2}W_2^2\left(\bfmu^N(T;\control),\hat\mu\right),
\end{equation}
where $\mu_0^N\in\probc(\R^d)$, $y_0\in\R^c$ are fixed initial conditions. Unfortunately, the Wasserstein distance is in general not differentiable at empirical measures, which prevents us from using the classical smooth \gls{acr:pmp} in the finite dimensional case.
This observation prompts us to study the optimal control problem in the mean-field regime: 
\begin{equation}\label{eqn: infinite dimensional OPC}\tag{\ensuremath{\mathbf{OCP_\infty}}}
    \inf_{\control\in \mathcal{U}}J_\infty(\control)
    =
    \inf_{\control\in \mathcal{U}}\frac{1}{2}W_2^2\left(\bfmu_\infty(T;\control),\hat\mu\right),
\end{equation}
where  $\mu_0^\infty$ is assumed to be absolutely continuous with respect to the Lebesgue measure on $\R^d$,  so that the Wasserstein-differentiability properties required in the previous section apply. Finally, we transfer the resulting mean-field controls to large but finite agent systems and assess their performance for \eqref{eqn: finite dimensional OPC}.

\subsection{Optimality conditions}

We start by specifying our main results to the splitting scenario.

\begin{theorem}[\gls{acr:pmp} for the splitting problem]\label{thrm:splitting gradient}
Given a control $\control^\star\in\U$ we denote by $\bfphi^\star_\infty$, $\bfmu^\star_\infty$ and $\bfy^\star_\infty$ the corresponding flow, curve in $\prob_c(\R^d)$ and controlled trajectory, respectively. The gradient of $J_\infty$ with respect to the control at $\control^\star$ is given by $\bfq^\star$,
where $(\bfp^\star,\bfq^\star)$ solves the backward adjoint equation
\begin{equation}\label{eqn: explicit adjoint}
\begin{split}
    -\partial_t p(t,x) =&-\int_{\R^d}DK(\zeta-\bfphi^\star_\infty(t,x))^\top p(t,x)\dd(\bfmu^\star_\infty(t))(\zeta)
    \\
    &-\frac{1}{M}\sum_{m=1}^M Df(\bfy_{\infty,m}^\star(t)-\bfphi^\star_\infty(t,x))^\top p(t,x)
    \\
    &+\int_{\R^d} DK\bigl(\bfphi^\star_\infty(t,\zeta)-\bfphi^\star_\infty(t,x)\bigr)^\top p(t,\zeta) \dd\mu_0(\zeta)
    \\
    -\dot q_m(t) =& \frac{1}{M}\int_{\R^d}Df(\bfy^\star_{\infty,m}(t)-\bfphi^\star_\infty(t,\zeta))^\top p(t,\zeta)\dd\mu_0(\zeta)
    \\
    &+\frac{1}{M}\sum_{j=1}^M Dg\bigl(\bfy^\star_{\infty,j}(t)-\bfy^\star_{\infty,m}(t)\bigr)^\top\bigl(q_j(t)-q_m(t)\bigr),
\end{split}
\end{equation}
with terminal condition $p(T)=\bfphi^\star_\infty(T,\cdot)-\mathcal T(\bfphi^\star_\infty(T,\cdot))$ and $q(T)=0$,
where $\mathcal T = \mathcal T_{\bfmu^\star_\infty(T)}^{\hat\mu}$ is the unique optimal transport map from $\bfmu^\star_\infty(T)$ to $\hat\mu$. Moreover, if $\control^\star$ is (locally) optimal then, for almost every $t\in[0,T]$, the following optimality condition holds
    \begin{equation*}
        \sum_{m=1}^M\bfq^\star_m(t)\cdot \control^\star _m(t)=\min_{\omega\in U}\sum_{m=1}^M\bfq_m^\star(t)\cdot \omega_m.
    \end{equation*}
\end{theorem}
\begin{proof}
The adjoint equation \eqref{eqn: explicit adjoint} is obtained from \eqref{eqn: abstract adjoint linearized system} through straightforward computations; see also \cref{proposition: differentiability of functionals} for the Wasserstein gradient of the vector field $F$. Given that $\bfmu^\star_\infty(T)$ is absolutely continuous with respect to the Lebesgue measure, we know from Proposition \ref{proposition: differentiability of functionals} that $\psi$ is Wasserstein differentiable at $\bfmu^\star_\infty(T)$ with Wasserstein gradient $\mathrm{Id}-\mathcal T$, where $\mathcal T$ is the unique optimal transport map between $\bfmu^\star_\infty(T)$ and $\hat\mu$. Therefore, the assumptions of \cref{thrm: PMP} are satisfied.
\end{proof}

\subsubsection{Comparison of \eqref{eqn: finite dimensional OPC} and \eqref{eqn: infinite dimensional OPC}}\label{subsec:performance_guarantees}

Here we discuss some compatibility results between the mean-field setting and the setting with the $N$-agent setting. Throughout this subsection we assume that the empirical initial measures are supported in $B(0,R)$ for some $R>0$ and that $\mu_0^N\to\mu_0^\infty$ in $W_2$. We first prove via $\Gamma$-convergence that optimal finite-dimensional controls converge (up to subsequences) to an optimal mean-field control.

\begin{proposition}
    If $(\control_N)\seq$ is a sequence of optimal controls for the finite-dimensional optimal control problem \eqref{eqn: finite dimensional OPC}, then there exists a subsequence that converges weakly to an optimal control for \eqref{eqn: infinite dimensional OPC}.
\end{proposition}
\begin{proof}
The terminal cost $\mu\mapsto \tfrac12 W_2^2(\mu,\hat\mu)$ is continuous on $(\mathcal P_2(\R^d),W_2)$ and, by
\cref{theorem: continuous dependence on controls}, the terminal state depends continuously on the control and on the initial
data. In particular, for every sequence $(\control_N)_{N\in\N}\subset \U$ converging weakly in $L^2([0,T];\R^c)$ to some
$\control^\star\in \U$, one has
\begin{equation}\label{eqn: gamma limit}
\lim_{N\to\infty}\frac12\,W_2^2\bigl(\bfmu^N(T;\control_N),\hat\mu\bigr)
=
\frac12\,W_2^2\bigl(\bfmu_\infty(T;\control^\star),\hat\mu\bigr).
\end{equation}

From \eqref{eqn: gamma limit} we obtain the $\Gamma$-convergence of $J_N$ to $J_\infty$ with respect to the weak $L^2$ topology
on $\U$. Indeed, if $\control_N\rightharpoonup \control^\star$ in $L^2$, then \eqref{eqn: gamma limit} yields
\begin{equation*}
    J_\infty(\control^\star)
    =\lim_{N\to\infty}J_N(\control_N)\leq \liminf_{N\to\infty}J_N(\control_N),
\end{equation*}
which is the $\Gamma$-$\liminf$ inequality. The $\Gamma$-$\limsup$ inequality follows by choosing, for any fixed
$\control\in \U$, the constant recovery sequence $\control_N\equiv \control$, so that
\begin{equation*} \limsup_{N\to\infty}J_N(\control_N)=\lim_{N\to\infty}J_N(\control)=J_\infty(\control).
\end{equation*}

Moreover, since $\U$ is weakly compact in $L^2$, the family $(J_N)_N$ is equi-coercive. Therefore, by the fundamental theorem of
$\Gamma$-convergence, any sequence of minimizers $(\control_N)_N$ admits a weakly convergent subsequence whose limit is a
minimizer of $J_\infty$, i.e.\ an optimal control for \eqref{eqn: infinite dimensional OPC}.
\end{proof}

Our second result is a performance guarantee when the mean-field control is applied to the finite-dimensional system. We note that the constant $C$ in the following statement arises from Grönwall estimates and thus depends exponentially on the terminal time, effectively limiting the practical utility of this bound to short horizons.

\begin{proposition}
There exists $C>0$, depending only on $R$, $U$, $\mathbf C$ and $T$, such that for every $\control\in\U$, $$\bigl|J_N(\control)-J_\infty(\control)\bigr|\le C\,W_2(\mu_0^N,\mu_0^\infty).$$
In particular, if $\control_\infty^\star$ minimizes $J_\infty$ and $\control_N^\star$ minimizes $J_N$, then $$0\leq J_N(\control_\infty^\star)-J_N(\control_N^\star)\leq 2C\,W_2(\mu_0^N,\mu_0^\infty).$$
\end{proposition}
\begin{proof}
Fix a control $\control\in\U$. By the a-priori support estimates in \cref{thrm: well posedness}, there is $\bar R $ such that
$\supp(\bfmu^N(T;\control)),\supp(\bfmu_\infty(T;\control))\subset B(0,{\bar R})$ for all $N\in\N$.
Let us introduce $a\coloneqq W_2\bigl(\bfmu^N(T;\control),\hat\mu\bigr)$ and $b\coloneqq W_2\bigl(\bfmu_\infty(T;\control),\hat\mu\bigr)$,
then $a,b\le 2\bar R$ and
\begin{equation*}
|J_N(\control)-J_\infty(\control)|
=\tfrac12|a^2-b^2|
=\tfrac12(a+b)|a-b|
\le 2\bar R\,W_2\bigl(\bfmu^N(T;\control),\bfmu_\infty(T;\control)\bigr).
\end{equation*}
Finally, \cref{theorem: Lipschitz dependence on initial data} yields
$W_2\bigl(\bfmu^N(T;\control),\bfmu_\infty(T;\control)\bigr)\le L\,W_2(\mu_0^N,\mu_0^\infty)$,
so the uniform estimate holds with $C\coloneqq 2\bar R L$. For the suboptimality gap we have
\begin{equation*}
    J_N(\control_\infty^\star)-J_N(\control_N^\star)
\le \bigl(J_N(\control_\infty^\star)-J_\infty(\control_\infty^\star)\bigr)
     +\bigl(J_\infty(\control_N^\star)-J_N(\control_N^\star)\bigr)
\le 2C\,W_2(\mu_0^N,\mu_0^\infty),
\end{equation*}
where we used $J_\infty(\control_\infty^\star)\le J_\infty(\control_N^\star)$ and the uniform bound above.
\end{proof}

\subsection{Link with the finite-dimensional \gls{acr:pmp} in the smooth case}\label{subsec:finite-dimensional PMP}

To clarify the relation between the finite-dimensional and mean-field \glspl{acr:pmp}, we consider a setting in which the classical \gls{acr:pmp} is well posed already at the particle level, namely a smooth terminal cost of expected-value type. Let $$\psi(\mu)=\int_{\R^d}\hat\psi(x)\dd\mu(x),$$
with $\hat{\psi}$ as in \cref{proposition: differentiability of functionals}. For the $N$-particle system, the terminal cost reads
\begin{equation}\label{eq:finite_dim_expected_value}
    \psi(\bfmu^N(T;\control))=\frac1N\sum_{n=1}^N \hat\psi(\bfx^{N,n}(T;\control)).
\end{equation}
In this case, the finite-dimensional optimal control problem \eqref{eqn: finite dimensional OPC} admits the classical PMP, with terminal adjoint condition $p_n(T)=\nabla \hat\psi(x^{N,n}(T)) / N$, $n=1,\dots,N$ and the corresponding first-order optimality condition for the control. After the natural $1/N$ rescaling of the non-controlled costates, the finite-dimensional adjoint system is consistent with the infinite-dimensional adjoint equation in~\cref{thrm: PMP} in the empirical distribution setting. In particular, passing to the mean-field limit in the (rescaled) finite-dimensional \gls{acr:pmp} yields the same PDE--ODE optimality system as deriving the \gls{acr:pmp} directly at the measure level.
For instance, in~\cite{bongini2017mean}, a mean-field \gls{acr:pmp} is obtained by first writing the smooth finite-dimensional \gls{acr:pmp} and then passing to the mean-field limit in the optimality system.
This strategy fails when the terminal cost is not differentiable on empirical measures, as in the case of the \mbox{Wasserstein distance.}

Our analysis is instead carried out directly at the level of probability measures: we derive first-order optimality conditions for the limiting PDE--ODE control problem without requiring the finite-dimensional approximations to admit a classical (smooth) \gls{acr:pmp}. This, in turn, allows us to work in greater generality and handle terminal costs such as the Wasserstein distance and more general optimal transport objectives. 
Note, moreover, that our measure-level \gls{acr:pmp}, unlike~\cite{bongini2017mean}, provides necessary conditions satisfied by \emph{all} mean-field optimal controls, including those that are not limits of particle-level optimal controls. Nonetheless, the performance of mean-field optimal controls can be related to large but finite multi-agent systems through the performance guarantees in \Cref{subsec:performance_guarantees}, which readily extend to terminal costs that are locally Lipschitz continuous in $W_2$.

\subsection{Algorithm and implementation}\label{subsec:algorithm}

To numerically solve the optimal control problem, we use gradient descent. We evaluate the gradient of the terminal cost with respect to the input via \cref{thrm:splitting gradient}, as we detail below. The pseudo-code for our algorithm is summarized in~\cref{alg:alg}. We intend this section as a proof-of-concept implementation of our results, and leave a rigorous convergence and numerical analysis to future work.

\paragraph{Computation of the gradient}
We adopt an optimize-then-discretize scheme: We discretize space and time to compute the analytic gradient, instead of computing gradients of discretized problems. We do so in three steps. First, we use the current control $\control$ and solve the forward equation~\eqref{eqn:mean-field-coupled-system} with a forward Euler scheme in time (with discretization $\Delta t>0$) and a conservative finite-volume discretization in space on Cartesian cells of side $\Delta x>0$, using a local Lax--Friedrichs \cite{leveque2002finite} numerical flux to approximate the divergence operator. Under a CFL condition, this monotone update preserves positivity and is conservative up to machine precision. When solving the system, we also keep track of the flow $\bfphi(t,\cdot)$; as a result, we obtain the trajectory $(\bfmu(t), \bfphi(t,\cdot), \bfy(t))_{t\in[0,T]}$ and, in particular, the terminal distribution $\bfmu(T)$.
Second, we compute the optimal transport map between the terminal distribution $\bfmu(T)$ and the target distributions $\hat\mu$; we will provide more details on this step when presenting our two-dimensional implementation. 
This way, we obtain the terminal condition for the backward equation.
Third, we solve the backward equation~\eqref{eqn: explicit adjoint} using a forward Euler scheme in time (again, with discretization $\Delta t>0$). This equation~\eqref{eqn: explicit adjoint} is parametric in $x$ and involves the computation of various ``spatial'' integrals; thus, we discretize the space in cubes of side $\Delta x$, and use the solution $(\bfmu(t), \bfphi(t,\cdot), \bfy(t))_{t\in[0,T]}$ obtained with the forward solver. This yields the costate $(\bfp(t,\cdot), \bfq(t))_{t\in[0,T]}$. By \cref{thrm:splitting gradient}, $(\bfq(t))_{t\in[0,T]}$ is the gradient of the terminal cost with respect to the control input.
Finally, we highlight that we solve the adjoint in $L^2$ without needing to resort to the solution of the adjoint equation in the Wasserstein space, as done in~\cite{pmpws} which ``doubles the dimension'' of the space, see also~\Cref{subsection: discussion}. Thus, solving the backward adjoint equation is ``as expensive as'' solving the forward equation.

\paragraph{Gradient descent}
Once the gradient is computed, we update the control input with a standard projected gradient descent with a step size of $\gamma>0$. The projection ensures that the control belongs to $\mathcal U$, which, here we take as the set of inputs whose agent-wise Euclidean norm does not, at all times, exceed $u_\mathrm{max}>0$.
To improve stability, we additionally apply an agent-wise normalization of the gradient so that each controlled agent has the same gradient magnitude. To ensure cost decrease we complement gradient descent with Armijo backtracking. We stop the algorithm when a maximum number of iterations is reached or when the terminal cost has not changed more than a given tolerance. 

\begin{algorithm}[h]
\small 
\caption{Projected gradient descent}\label{alg:alg}
\begin{algorithmic}[1] 
    \Require $T>0$ (time horizon), $\Delta x, \Delta t > 0$ (discretization), $\hat\mu$ (target distribution), $\mu_0,\bfy_0$ (initial condition), $u_\mathrm{max}$ (maximum control input), $\gamma>0$ (step size), stopping criteria (e.g., tolerance), $\control$ initial control
    \Require $\Call{Forward}{}, \Call{Backward}{}$ (solvers), $\Call{OptimalTransportMap}{}$ (OT map), $\Call{NormalizeGradient}{}$ (optional gradient norm.), $\Call{Project}{}$ (projector)
    \While{not converged}
        \State \textbf{\# compute gradient}
        \State $(\bfmu(t),\bfphi(t,\cdot),\bfy(t))_{t\in\{0,\Delta t, \ldots\}} \gets \Call{Forward}{\control, \mu_0, \bfy_0, T, \Delta t, \Delta x}$
        \State $\mathcal T\gets \Call{OptimalTransportMap}{\bfmu(T),\hat\mu,\Delta x}$
        \State $(\bfp(t,\cdot), \bfq(t))_{t\in\{0,\Delta t, \ldots\}} \gets \Call{Backward}{(\bfmu(t),\bfphi(t,\cdot),\bfy(t))_{t\in\{0,\Delta t, \ldots\}}, \mathcal T, T, \Delta t, \Delta x}$
        \State \textbf{\# gradient step}
        \State $(\bfq(t))_{t\in\{0,\Delta t, \ldots\}} \gets \Call{NormalizeGradient}{(\bfq(t))_{t\in\{0,\Delta t, \ldots\}}}$\Comment{optional grad. norm.}
        \State $\control \gets \Call{Project}{\control - \gamma\bfq, u_\mathrm{max}}$
        \If{converged}
            \State break 
        \EndIf
    \EndWhile
    \State \textbf{return} $\control$ 
\end{algorithmic}
\end{algorithm}

\paragraph{Our implementation}
We implement this scheme for a two-dimensional example. 
We use $\Delta t=0.005$ and $\Delta x=0.05$.
We use a target configuration with two target points, each associated with half of the mass of the distribution. For this case, the computation of the optimal transport map amounts to identifying the line (given its slope) that divides the distribution into two of equal mass (e.g., see \cite[Exercise A.3]{figalli2021invitation}). This can be done efficiently using a standard bisection scheme.
Finally, we use a step size of 0.1. 
Our implementation is in Python using JAX; open-source code can be found at \href{https://github.com/nicolaslanzetti/wasserstein-sparse-optimal-control}{\texttt{https://github.com/nicolaslanzetti/wasserstein-sparse-optimal-control}}.

\subsection{Numerical case study}\label{subsec:case study}

\paragraph*{General setting} 

We consider a setting where the control goal is, within a time horizon of $T=1.5$, to split a population $\mu_0$ of agents whose state is their position in the plane, into two of equal mass at the target locations $(0,-1)$ and $(0,1)$. The population is initially distributed according to a Gaussian centered at the origin, with standard deviation $1.2$, and truncated so that all the mass lies in the ball of radius $0.8$; see~\cref{fig:optimal_trajectories}.
The vector fields in dynamics~\eqref{eqn:mean-field-coupled-system} are $K(z)=-h_{k_{\mu,\mathrm{a}},\sigma_{\mu,\mathrm{a}}}(\norm{z})z+h_{k_{\mu,\mathrm{r}},\sigma_{\mu,\mathrm{r}}}(\norm{z})z$ (we take $\sigma_{\mu,\mathrm{r}}<\sigma_{\mu,\mathrm{a}}$, so that the interaction $K$ is repulsive at short ranges and attractive at larger distances.),  $f(z)=h_{k,\sigma}(\norm{z})z$ (repulsive), and $g(z)=-h_{k_y,\sigma_y}(\norm{z})z$ (attractive), where $h_{k,\sigma}(d)=-k\exp(-\frac{d^2}{2\sigma^2})$ is a Gaussian kernel. 
For our experiments, we used
$k_{\mu,\mathrm{a}}=3$,
$\sigma_{\mu,\mathrm{a}}=0.25$,
$k_{\mu,\mathrm{r}}=30$,
$\sigma_{\mu,\mathrm{r}}=0.1$,
$k=22$,
$\sigma=0.325$
$k_y=30$,
$\sigma_y=0.1$.
There are $M=6$ controlled agents with the initial position $y_0$ as in~\cref{fig:optimal_trajectories}, with controls bounded by $u_\mathrm{max}=1$.
We compute (locally) optimal controls using \cref{alg:alg}, starting from the initial condition $\control$.
On standard commodity hardware (Intel Core i7-12700H (14C/20T) CPU), the computation takes approximately $10$ minutes for $12$ iterations, i.e., less than $60$ seconds per forward--backward sweep (including the computation of the optimal transport map). Across all runs, we did not observe numerical instabilities. As a practical diagnostic, we monitored the Courant number, which remained below $0.5$ along the optimal trajectory.

\begin{figure}[t]
    
    \centering
    \includegraphics[width=0.24\textwidth]{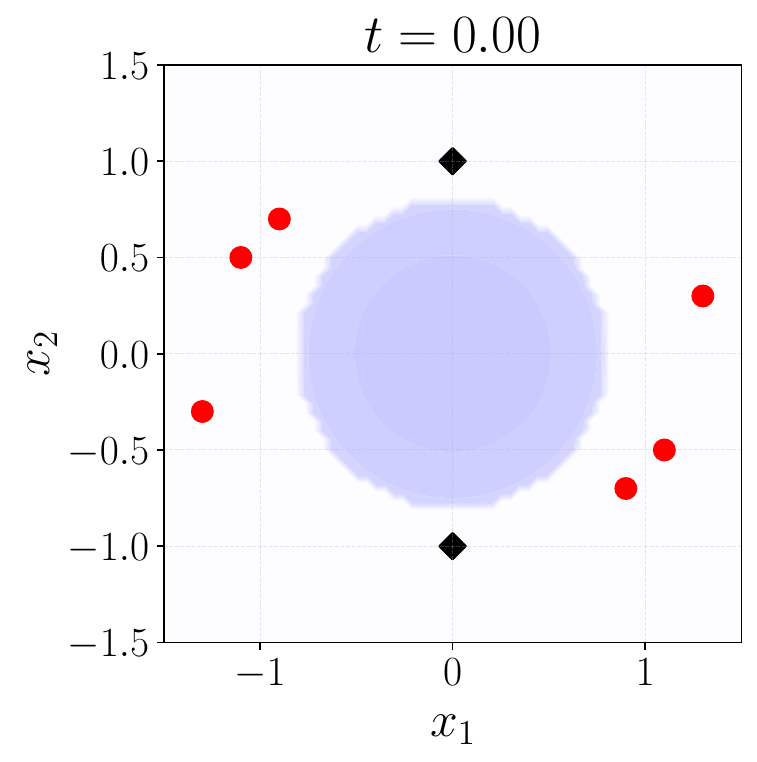}
    \includegraphics[width=0.24\textwidth]{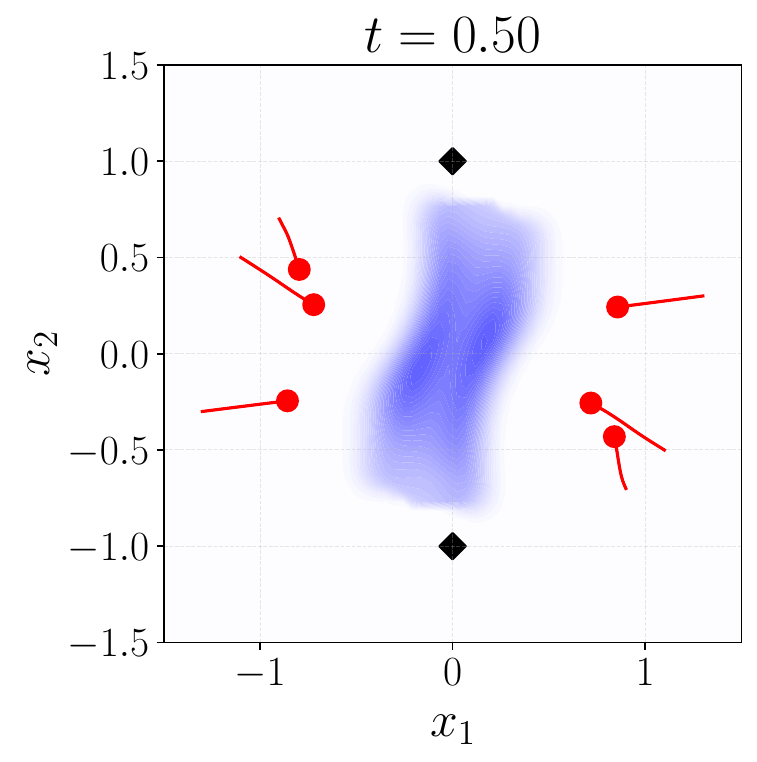}
    \includegraphics[width=0.24\textwidth]{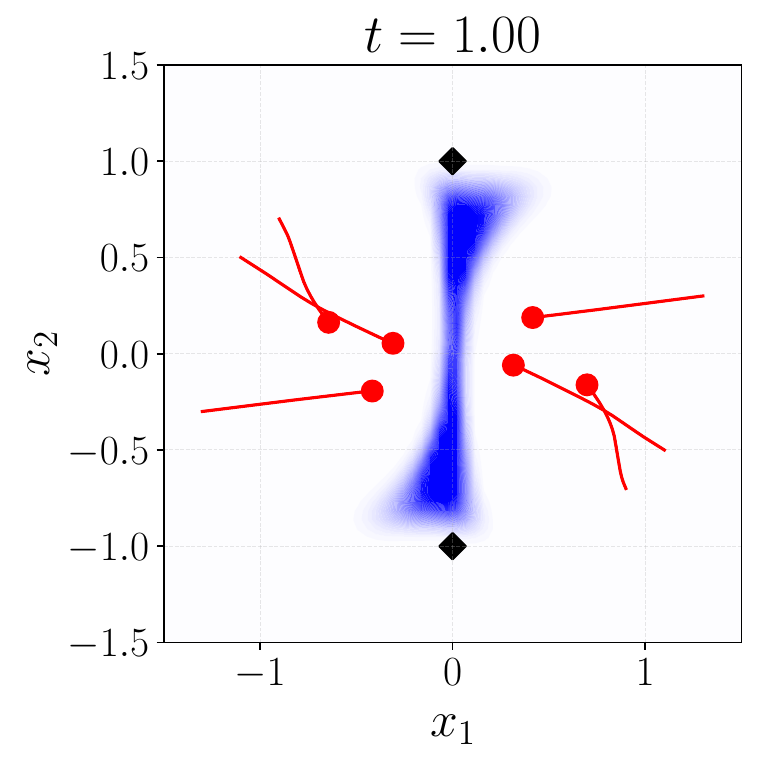}
    \includegraphics[width=0.24\textwidth]{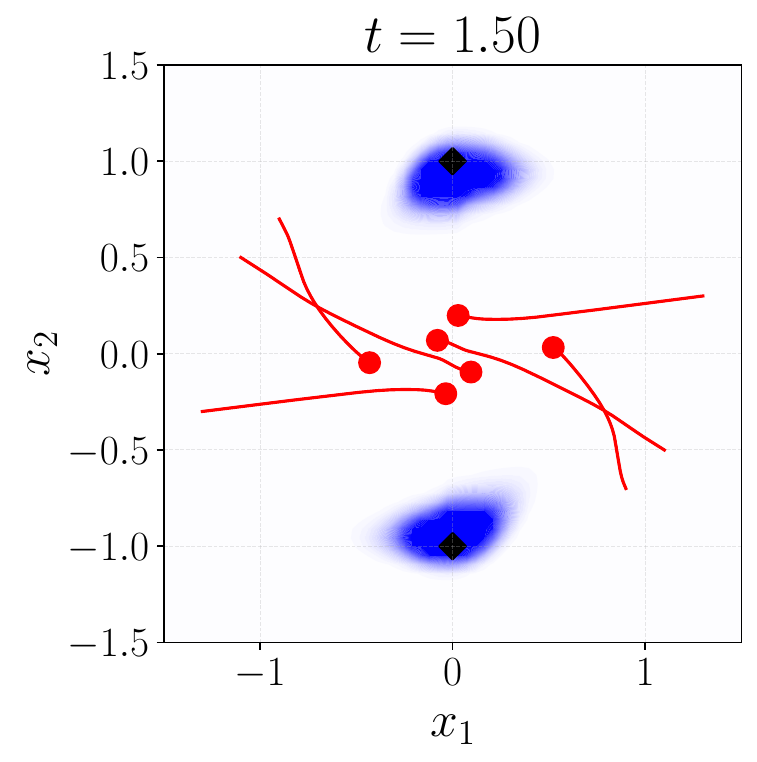}
    \vspace{-2ex}
    \caption{Optimal trajectories of the controlled agents (red circles) that split the distribution (in blue) into two over the desired targets (black diamonds).
    }
    \label{fig:optimal_trajectories}
    \vspace{-2ex}
\end{figure}

\paragraph{Optimal trajectories}
We show the (locally) optimal evolution of the distribution and the (locally) optimal trajectories of the steering agents in~\cref{fig:optimal_trajectories}. We observe that the agents are capable of splitting the distribution into two of equal mass and steer the agents to the desired location. Despite being severely underactuated and having no prior ($\control$ is initialized to zero), with few agents steering a full distribution, the terminal distribution matches the target distribution well.

\paragraph{PMP residual and terminal cost}
\begin{wrapfigure}{r}{0.33\textwidth}
    \vspace{-\baselineskip}
    \centering
    \def\xmaxplot{12}

\begin{tikzpicture}
    \begin{groupplot}[
        group style={
            group size=1 by 2,
            vertical sep=0.2cm,
            xlabels at=edge bottom,
            xticklabels at=edge bottom
        },
        width=0.7\linewidth,
        height=0.5\linewidth,
        scale only axis,
        xmin=1, xmax=\xmaxplot,
        xtick distance=2,          
        ylabel near ticks, 
        xlabel near ticks,
        grid=major,
        every axis plot/.append style={ultra thick} 
    ]

    \nextgroupplot[
        ylabel={PMP residual},
        ylabel style={font=\footnotesize},
        tick label style={font=\scriptsize},
    ]
    \addplot[solid, violet] table [x expr=\coordindex, y index=0] {results/plots_mf/pmp_residual_diagnostics.csv};

    \nextgroupplot[
        ylabel={Control cost},
        xlabel={Iterations},
        ylabel style={font=\footnotesize},
        xlabel style={font=\footnotesize},
        tick label style={font=\scriptsize}
    ]
    \addplot[solid, teal] table [x expr=\coordindex, y index=0] {results/plots_mf/loss_diagnostics.csv};

    \end{groupplot}
\end{tikzpicture}
\end{wrapfigure}
To assess the quality of our solution, we evaluate the \gls{acr:pmp} residual and the terminal cost (i.e., squared Wasserstein distance to the targets at the terminal time) as a function of the iterations of gradient descent. The \gls{acr:pmp} residual is defined, using~\cref{thrm: PMP}, as the relative $L^2$ distance between $\bfq \cdot \control$ and $\min _{\omega\in U}\bfq\cdot \omega$, where $\bfq$ is the costate calculated with the current control $\control$. 
As shown on the right, both the \gls{acr:pmp} residual (blue solid line, left axis) and the terminal cost (red dotted line, right axis) decrease monotonically with the iterations of our gradient descent scheme.
The \gls{acr:pmp} residual is computed on the continuous problem, and so, because of time and space discretization, is not expected to converge close to 0.

\paragraph{Performance in the large-scale but finite system}
To conclude, we evaluate the performance of the mean-field optimal control in a setting where the population consists of a large but finite number of agents. We consider $N=500$ agents, sampled at random from the initial distribution $\mu_0$, and apply the (locally) optimal control obtained above in the mean-field regime.
\begin{figure}[t]
    \centering
    \includegraphics[width=0.24\textwidth]{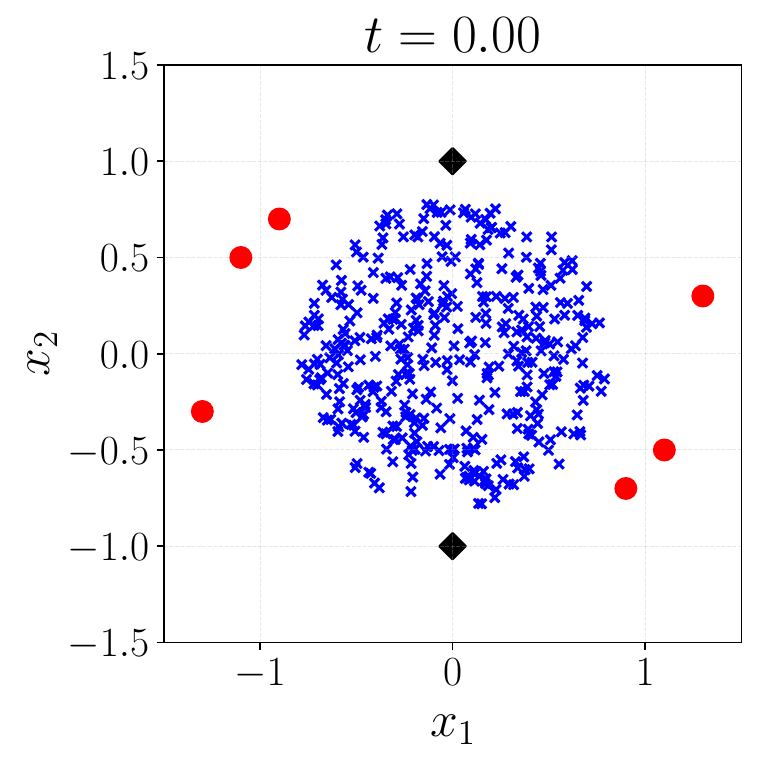}
    \includegraphics[width=0.24\textwidth]{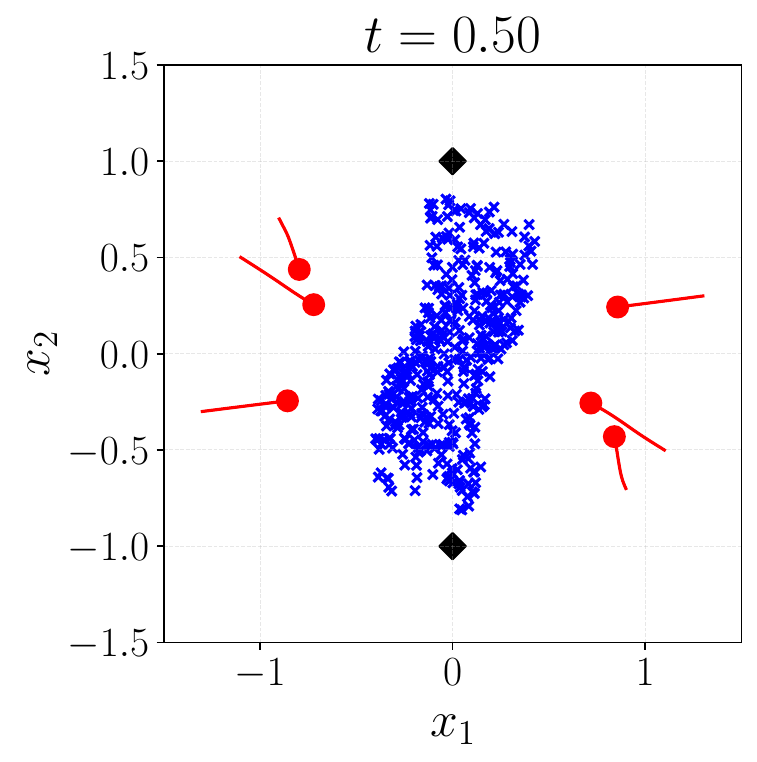}
    \includegraphics[width=0.24\textwidth]{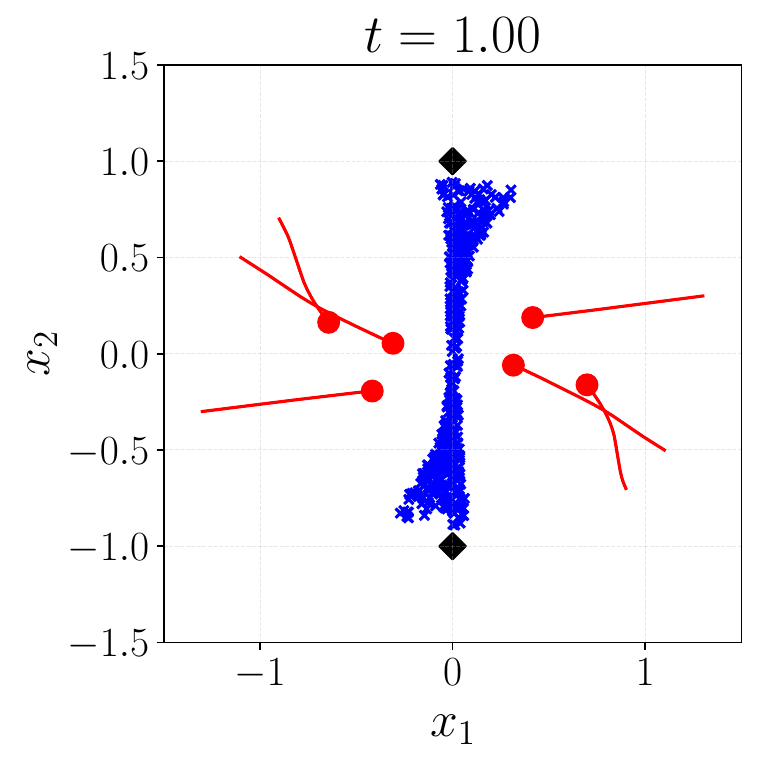}
    \includegraphics[width=0.24\textwidth]{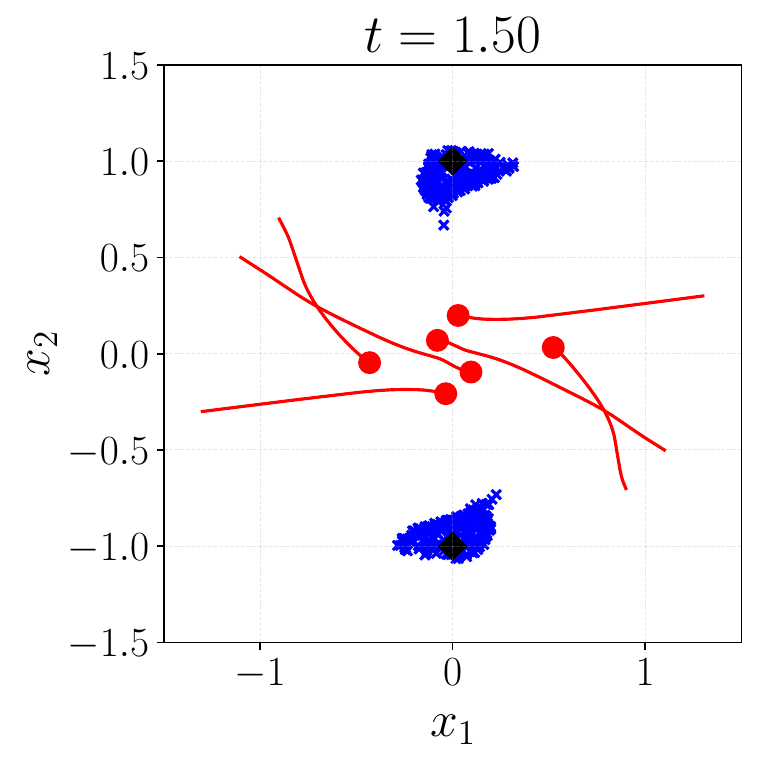}
    \vspace{-2ex}
    \caption{Trajectories in the case of finitely many agents. The behavior closely mimics the one in the mean-field, with successful ``splitting'' of the distribution.}
    \label{fig:finite}
    \vspace{-2ex}
\end{figure}
As shown in~\cref{fig:finite}, the mean-field optimal control successfully splits the population of agents, achieving a terminal cost\footnote{Given the random sampling of the population, we average the results over 50 seeds. The standard deviation is $0.027$.} (i.e., squared Wasserstein distance divided by 2) of $0.050$ which closely matches the $0.023$ obtained in the mean-field regime. 

\bibliographystyle{plain}
\bibliography{bibliography}

\appendix
\section{Technical proofs}\label{sec: technical proofs}

In this appendix, we prove \cref{propo: first order expansion}. We start with a
standard integral-remainder formula for Wasserstein gradients.

\begin{proposition}\label{propo:integral-remainder}
Let $\varphi\colon\prob_2(\R^d)\to\R$, $\mu^\star\in\prob_c(\R^d)$, and $w\in L^\infty(\R^d,\R^d;\mu^\star)$.
Set $\nu^\star=(\Id+w)_\#\mu^\star$ and $\mu^\theta=(\Id+\theta w)_\#\mu^\star$ for $\theta\in[0,1]$.
Assume that $\varphi$ is Wasserstein differentiable and that $(\mu,\zeta)\mapsto D_\mu\varphi[\mu](\zeta)$ is continuous.
Then
\begin{equation*}
    \varphi(\nu^\star)-\varphi(\mu^\star)
    =\int_{\R^d} D_\mu\varphi[\mu^\star](\zeta)\cdot w(\zeta)\,d\mu^\star(\zeta)
    +R(\mu^\star,w),
\end{equation*}
with
\begin{equation*}
    R(\mu^\star,w)
    =\int_0^1\!\int_{\R^d}
    \Bigl(D_\mu\varphi[\mu^\theta]\bigl(\zeta+\theta w(\zeta)\bigr)-D_\mu\varphi[\mu^\star](\zeta)\Bigr)\cdot w(\zeta)\,d\mu^\star(\zeta)\,d\theta.
\end{equation*}
\end{proposition}

\begin{proof}
Define $\Phi(\theta):=\varphi(\mu^\theta)$. By Wasserstein differentiability and \cref{proposition: Wasserstein gradients are strong} $\Phi$ is $C^1$ since $    \Phi'(\theta)=\int_{\R^d} D_\mu\varphi[\mu^\theta]\bigl(\zeta+\theta w(\zeta)\bigr)\cdot w(\zeta)\dd\mu^\star(\zeta)$,
which is continuous due to the continuity assumptions and dominated convergence theorem. Integrating $\Phi'$ on $[0,1]$ and adding/subtracting $\int_{\R^d} D_\mu\varphi[\mu^\star](\zeta)\cdot w(\zeta)\dd\mu^\star(\zeta)$
yields the claim. 
\end{proof}
To keep the notation compact, we write here the time dependence with a subscript.

\begin{proof}[Proof of \cref{propo: first order expansion}]
We prove the expansion for $F$; the proof for $G$ is analogous. 
Let $M:=\max_{s\in[t_0,T]}\|\mathfrak w(s)\|_{X(\mu_0)}$. 
Therefore, since $\supp\mu_0$ is compact and $(t,x)\mapsto \bfphi^\star(t,x)$ is continuous, there exists a compact set $K_x\subset\R^d$
such that for $\varepsilon$ small enough, every $s\in[t_0,T]$, every $\theta\in[0,1]$, and $\mu_0$-a.e.\ $x$, $\bfphi_s^\star(x)+\theta\varepsilon \bfw_s(x)\in K_x$ and similarly $\bfy_s^\star+\theta\varepsilon \bfv_s\in K_y\subset\R^c$ for a compact $K_y$.
Hence the relevant measures remain in a compact set $K_\mu\subset\prob_2(\R^d)$.
Since the derivatives of $F$ are continuous (assumption \textnormal{H2)}), they are uniformly continuous on
$K_\mu\times K_x\times K_y$ (and $K_\mu\times K_x\times K_y\times K_x$ for $D_\mu F$); denote by $\omega$ a modulus of continuity
(valid for all such derivatives) with $\omega(r)\to 0$ as $r\to 0$.

\noindent\emph{Step 1: decomposition.}
For $s\in[t_0,T]$ and $\mu_0$-a.e.\ $x$, set
\begin{equation*}
    F_s^\varepsilon(x):=F[\bfrho_s^\varepsilon]\bigl(\bfpsi_s^\varepsilon(x),\bfz_s^\varepsilon\bigr),
\qquad
F_s^\star(x):=F[\bfmu_s]\bigl(\bfphi_s(x),\bfy_s\bigr),
\end{equation*}
and split $    F_s^\varepsilon(x)-F_s^\star(x)=A_s^\varepsilon(x)+B_s^\varepsilon(x)$,
where
\begin{align*}
    A_s^\varepsilon(x)\coloneqq &F[\bfrho_s^\varepsilon]\bigl(\bfpsi_s^\varepsilon(x),\bfz_s^\varepsilon\bigr)
                 -F[\bfrho_s^\varepsilon]\bigl(\bfphi_s(x),y_s\bigr),\\
B_s^\varepsilon(x)\coloneqq &F[\bfrho_s^\varepsilon]\bigl(\bfphi_s(x),\bfy_s\bigr)
                 -F[\bfmu_s]\bigl(\bfphi_s(x),\bfy_s\bigr).
\end{align*}

\noindent\emph{Step 2: first-order expansion of $A_s^\varepsilon$.}
For fixed $s$, the map $(x,y)\mapsto F[\bfrho_s^\varepsilon](x,y)$ is $C^1$, so a first-order Taylor expansion between
$(\bfphi_s(x),y_s)$ and $(\bfpsi_s^\varepsilon(x),\bfz_s^\varepsilon)$ gives
\begin{equation*}
    A_s^\varepsilon(x)
=\varepsilon\,D_yF[\bfrho_s^\varepsilon]\bigl(\bfphi_s(x),\bfy_s\bigr)\bfv_s
+\varepsilon\,D_xF[\bfrho_s^\varepsilon]\bigl(\bfphi_s(x),\bfy_s\bigr)\bfw_s(x)
+r_{xy}^\varepsilon(s,x),
\end{equation*}
with a remainder satisfying
\begin{equation*}
    \sup_{s\in[t_0,T]}\frac{\|r_{xy}^\varepsilon(s,\cdot)\|_{L^2(\mu_0)}}{\varepsilon}
\le C\,\omega(C\varepsilon)\xrightarrow[\varepsilon\to 0]{}0.
\end{equation*}
Moreover, since $W_2(\bfrho_s^\varepsilon,\bfmu_s)\leq \|\bfpsi_s^\varepsilon-\bfphi_s\|_{L^2(\mu_0)}=\varepsilon\|\bfw_s\|_{L^2(\mu_0)}\leq \varepsilon M$,
uniform continuity yields the basepoint replacement
\begin{equation*}
    D_{x,y}F[\bfrho_s^\varepsilon]\bigl(\bfphi_s(x),\bfy_s\bigr)
= D_{x,y}F[\bfmu_s]\bigl(\bfphi_s(x),\bfy_s\bigr) + o_\varepsilon(1)
\quad\text{uniformly in }(s,x),
\end{equation*}
and therefore, uniformly in $s\in[t_0,T]$, 
\begin{equation*}
    A_s^\varepsilon(x)
=\varepsilon\,D_yF[\bfmu_s]\bigl(\bfphi_s(x),\bfy_s\bigr)\bfv_s
+\varepsilon\,D_xF[\bfmu_s]\bigl(\bfphi_s(x),\bfy_s\bigr)\bfw_s(x)
+o(\varepsilon)
\quad\text{in }L^2(\mu_0).
\end{equation*}

\noindent\emph{Step 3: Wasserstein expansion of $B_s^\varepsilon$.}
Consider $\varphi_{s,x}(\mu)\coloneqq F[\mu]\bigl(\bfphi_s(x),\bfy_s\bigr)$ for $s$ and $x$ fixed. Applying
\cref{propo:integral-remainder} to $\varphi_{s,x}$ at $\bfmu_s$ along the perturbation $\bfrho_s^\varepsilon=(\bfphi_s+\varepsilon \bfw_s)_\#\mu_0$ gives
\begin{equation*}
    B_s^\varepsilon(x)
=\varepsilon\int_{\R^d} D_\mu F[\bfmu_s]_{(\bfphi_s(x),\bfy_s)}\bigl(\bfphi_s(\zeta)\bigr)\cdot \bfw_s(\zeta)\,d\mu_0(\zeta)
+r_\mu^\varepsilon(s,x),
\end{equation*}
with
\begin{equation*}
    \sup_{s\in[t_0,T]}\frac{\|r_\mu^\varepsilon(s,\cdot)\|_{L^2(\mu_0)}}{|\varepsilon|}
\le C\,\omega(C\varepsilon)\xrightarrow[\varepsilon\to 0]{}0.
\end{equation*}
Combining the expansions of $A_s^\varepsilon$ and $B_s^\varepsilon$ yields the desired first-order formula with a remainder
$R_F^\varepsilon$ satisfying $\sup_s \|R_F^\varepsilon(s,\cdot)\|_{L^2(\mu_0)}/\varepsilon\to 0$.
\end{proof}

\end{document}